\documentclass{article}
\usepackage{amsmath,amssymb,amsthm}
\usepackage{graphicx,subfigure,float,url,color}
\usepackage{mathrsfs}
\usepackage[colorlinks=true]{hyperref}
\usepackage{pdfsync}

\hypersetup{linkcolor=blue,urlcolor=blue,citecolor=red}

\theoremstyle{plain} 
\newtheorem{theorem}{Theorem}

\newtheorem{lemma}{Lemma}

\theoremstyle{definition}  

\newtheorem{example}{Example}

\newtheorem{remark}{Remark}
\numberwithin{equation}{section}
\renewcommand{\leq}{\leqslant}
\renewcommand{\geq}{\geqslant}

\title{
\textbf{Steady-state and periodic exponential turnpike property for optimal control problems in Hilbert spaces}
}
\begin{document}

\author{
Emmanuel Tr\'elat
\thanks{
Sorbonne Universit\'es, UPMC Univ Paris 06, CNRS UMR 7598, Laboratoire Jacques-Louis Lions, Institut Universitaire de France, F-75005, Paris, France. (emmanuel.trelat@upmc.fr)}
\quad\quad
Can Zhang
\thanks{School of Mathematics and Statistics, Wuhan University, 430072 Wuhan, China; Sorbonne Universit\'es, UPMC Univ Paris 06, CNRS UMR 7598, Laboratoire Jacques-Louis Lions, F-75005 Paris, France.
(zhangcansx@163.com)}
\quad\quad
Enrique Zuazua
\thanks{
DeustoTech, Fundaci\'on Deusto, Avda Universidades, 24, 48007, Bilbao, Basque Country, Spain;
Departamento de Matem\'aticas,
Universidad Aut\'onoma de Madrid,
28049 Madrid, Spain;
Facultad Ingenier\'ia, Universidad de Deusto, Avda. Universidades, 24, 48007 Bilbao, Basque Country, Spain;
Sorbonne Universit\'es, UPMC Univ Paris 06, CNRS UMR 7598, Laboratoire Jacques-Louis Lions,  F-75005, Paris, France.
 (enrique.zuazua@uam.es)}
}

\date{}
\maketitle

\begin{abstract}
In this work, we  study the steady-state (or periodic) exponential turnpike property of optimal control problems in Hilbert spaces. The turnpike property, which is essentially due to the hyperbolic feature of the Hamiltonian system resulting from the Pontryagin maximum principle, reflects the fact that, in large time, the optimal state, control and adjoint vector remain most of the time close to an optimal steady-state. A similar statement holds true as well when replacing an optimal steady-state by an optimal periodic trajectory.
To establish the result, we design an appropriate dichotomy transformation, based on solutions of the algebraic Riccati and Lyapunov equations. 
We illustrate our results with examples including linear heat and wave equations with periodic tracking terms.
\end{abstract}

\textbf{Keywords:} Exponential turnpike property, periodic tracking, periodic optimal controls, stability analysis, dichotomy transformation. \\

\textbf{AMS subject classifications:} 49J20, 49K20, 93D20.

\section{Introduction}

The turnpike property of optimal trajectories was firstly observed and investigated by  economists for  finite-dimensional  discrete-time optimal control problems (see, e.g., \cite{Mc}).  
The turnpike property reflects the fact that, for an optimal control problem for which the time horizon is large enough, its optimal solution spends most of the time remaining close to a referred turnpike, which is usually the optimal solution of a corresponding ``static" optimal control problem.
In the last decades, several turnpike theorems for  optimal control problems  have been obtained in a large number of works (see, for instance, \cite{Kokotovic, AL, CHJ, CarlsonBOOK, Grune1,
DorfmanSamuelsonSolow, Faulwasser1,G3, GTZ, LW, Mc, PZ1, PZ2, Rapaport, TZ1, Z2, Z3, Z4} and references therein), for discrete-time or continuous-time problems involving control systems in finite dimension.

The usual turnpike property is somehow a qualitative feature for the limiting structure of optimal solutions to the optimal control problem as the time horizon tends to infinity. 
A quantitative behavior of the turnpike property (see, e.g., \cite{Grune1,PZ1,TZ1}) is called the \emph{exponential turnpike property} if the optimal solution remains exponentially close to the referred turnpike for a sufficiently large time interval contained in the time horizon frame.
We also mention that the authors of \cite{CLLP} proved that the long-time average of the solution for a mean field game system converges exponentially to the solution of the associated stationary ergodic mean field game.

In the recent work \cite{PZ1},  the exponential turnpike property has been established for linear quadratic optimal control problems in finite dimension under the Kalman controllability rank condition, as well as for linear infinite-dimensional systems, covering the cases of linear heat and wave equations with internal controls under some observability  inequality assumptions. A local version, for semilinear heat equations, has been obtained in \cite{PZ2}.  

The authors of \cite{TZ1} established the (local) exponential turnpike property for general  finite-dimensional nonlinear control systems with general terminal constraint conditions, under some appropriate controllability and smallness assumptions. Not only the optimal state and control, but also the corresponding adjoint vector, resulting from the application of the Pontryagin maximum principle, were shown to remain exponentially close to an extremal triple for a corresponding static optimal control problem, except at the extremities of the time frame.  The main ingredient in \cite{TZ1} is an exponential dichotomy transformation early established in \cite{WK} to uncouple the two-point boundary value problem coming from the Pontryagin maximum principle, reflecting the hyperbolicity feature of the Hamiltonian system.

The objective of the present paper is to establish the exponential turnpike property 
for general infinite-dimensional nonlinear optimal control problems under exponential stabilizability and detectability assumptions (Theorem \ref{731-10} in Section \ref{sub2}). This extends to Hilbert spaces the main results of \cite{TZ1}. 
The result implies that, except at the beginning and at the end of the time frame, the optimal trajectory remains exponentially close to a steady-state, which is itself characterized as being a minimizer for an associated ``static optimal control problem". We stress that, as in \cite{TZ1}, our result establishes the exponential turnpike property, not only for the optimal state, but also for the optimal control and for the adjoint state (or costate) coming from the application of the Pontryagin maximum principle. The latter property is particularly useful in order to implement and initialize successfully a numerical shooting method.

As a second main result (Theorem \ref{periodictarget} in Section \ref{sec_periodic}), we consider linear quadratic optimal control problems with periodic tracking trajectories, i.e., linear  autonomous control systems (still in Hilbert spaces) with a quadratic cost in which the integrand involves a periodic tracking term.
We prove that, under exponential stabilizability and detectability assumptions, the optimal trajectory (also, control and adjoint state) remains exponentially close, except at the beginning and the end of the time frame, to a periodic optimal trajectory, which is characterized as being the optimal solution of an associated periodic optimal control problem.
We are not aware of any general result establishing such a periodic turnpike property, even in the finite-dimensional case. Note however that Samuelson, who is, by the way, the inventor of the turnpike phenomenon that he discovered in the context of a Von Neumann model in view of deriving efficient programs of capital accumulation (see \cite[Chapter 12]{DorfmanSamuelsonSolow}), established in \cite{Samuelson1976} a periodic turnpike property for a specific optimal growth problem in economics in which the integrand of the minimization functional is periodic. Periodic turnpike has been also considered in the recent paper \cite{Zanon} within the dissipativity context.

To prove the results, our approach takes advantage of the hyperbolic feature of optimality systems (see \cite{Rockafellar1973}) resulting from the Pontryagin maximum principle, as in \cite{TZ1}. However, the invertibility of solutions of the matrix algebraic Riccati equation played an important role in the argument of \cite{TZ1}, but to the best of our knowledge, this argument is in general not valid in the infinite dimensional setting, because invertibility is closely related to an exact observability inequality and thus this would be a too much restrictive assumption in view of applications. One of the main technical novelties of the present paper is to design an extensive dichotomy transformation to overcome this difficulty.

The paper is organized in the following way.  
In Section~\ref{main}, we present our main results and some applications.
In Theorem~\ref{731-10}, we establish the exponential turnpike property for general nonlinear autonomous optimal control problems in a Hilbert space, under appropriate stability assumptions and smallness conditions.
In Theorem~\ref{periodictarget}, we establish the exponential turnpike property for linear quadratic optimal control problems with periodic tracking terms, in which the referred turnpike  is a periodic optimal solution for a periodic optimal control problem. 
Sections~\ref{1252} and \ref{1251} are devoted to the proofs of the main results.

\section{Main results}\label{main}
Throughout the paper, given a Hilbert space $Z$, we denote by $\langle\cdot,\cdot\rangle_Z$ the usual inner product and by $\|\cdot\|_Z$ the corresponding norm. The notation $L(X,Y)$ designates the space of bounded linear operators from the Hilbert space $X$ to the Hilbert space $Y$. 

\subsection{Steady-state exponential turnpike for nonlinear optimal control problems}\label{sub2}
Let $X$ and $U$ be two Hilbert spaces, which are accordingly identified with their duals. 
We define hereafter the dynamical optimal control problem $(OCP^T)$, and then the corresponding static optimal control problem $(P_s)$ yielding the optimal steady-state around which the turnpike is expected. The exponential turnpike result is then stated in Theorem \ref{731-10}.

\paragraph{The dynamical optimal control problem $(OCP^T)$.}
For every $T>0$ and every $y_0\in X$, we consider  the optimal control problem
$$(OCP^T)\qquad\qquad\inf_{u(\cdot)\in L^2(0,T;U)}\;\;\;\;J^T(u(\cdot))= \int_0^Tf^0(y(t),u(t))\,dt,$$
where $y(\cdot)\in C([0,T];X)$ is the mild  solution\footnote{Recall that the form of mild solution is 
$$
y(t)=e^{At}y_0+\int_0^te^{A(t-s)}f(y(s),u(s))\,ds,\;\;\;\;t\in[0,T],
$$
where $e^{At}$ is the $C_0$ semigroup in $X$ with generator $ A: D(A)\subset X \rightarrow X$.} of
\begin{equation}\label{731-9}
\left\{
\begin{split}
&\dot{y}(t)=Ay(t)+f(y(t),u(t)),\;\;\;\text{for a.e.}\;\;t\in[0,T],\\
& y(0)=y_0,
\end{split}\right.
\end{equation}
corresponding to the control function $u(\cdot)\in L^2(0,T;U)$.
Here, $A:D(A)\subset X\rightarrow X$ is a linear (unbounded) operator generating a $C_0$ semigroup on $X$, and the function $f^0:X\times U\rightarrow \mathbb{R}$ and the mapping $f:X\times U\rightarrow X$ are assumed to be twice continuously 
Fr\'echet differentiable and globally Lipschitz continuous with respect to $y$ for each $u\in U$.

Existence of optimal controls for the problem $(OCP^T)$ is a classical issue (see, e.g., \cite[Chapter 3]{LiXunjing}) and is generally ensured under adequate convexity assumptions. Here, we assume that the problem $(OCP^T)$ exists  at least one optimal solution. Let $(y^T(\cdot),u^T(\cdot))$ be any of them.
According to the Pontryagin maximum principle in a Hilbert space (see \cite[Chapter 4]{LiXunjing}), there exists $\lambda^T(\cdot)\in C([0,T],X)$, called adjoint state or costate, such that
\begin{equation}\label{721-1}
\left\{
\begin{split}
\dot{y}^T(t)&=Ay^T(t)+H_\lambda(y^T(t),\lambda^T(t),u^T(t)),\;\;\;\;\;\;\;y^T(0)=y_0,\\
\dot{\lambda}^T(t)&=-A^*\lambda^T(t)-H_y(y^T(t),\lambda^T(t),u^T(t)),\;\;\lambda^T(T)=0,
\end{split}
\right.
\end{equation}
in the mild sense along $[0,T]$ and
\begin{equation}\label{721-2}
H_u(y^T(t),\lambda^T(t),u^T(t))=0,\;\;\;\text{for a.e.}\;\;t\in[0,T] ,
\end{equation}
where $A^*$ is the adjoint operator associated with $A$, with the domain $D(A^*)$, and
\begin{equation}\label{def_H}
H(y,\lambda,u)=\langle \lambda, f(y,u)\rangle_X -f^0(y,u).
\end{equation}
is the (normal) Hamiltonian of the optimal control problem. The index at $H$ above designates the partial derivative.

Note that, in $(OCP^T)$, we have assume that the final state $y(T)$ is free. This is for two reasons. The first is that this implies that there is no abnormal minimizer\footnote{In the general  statement  of Pontryagin maximum principle for optimal control problems, the extremal lift  is said to be {\it abnormal} (resp. {\it normal}) whenever the Lagrange multiplier associated with the cost is zero (resp. {\it nonzero}); See, e.g., \cite{TZ1} for more details.}. 
The second, which is the main one, is that, if $y(T)$ were to be fixed to some prescribed point $y_1\in H$, then the Pontryagin maximum principle would fail in general (see \cite[Chapter 4]{LiXunjing}). We refer the reader 
to Section~\ref{con} for further comments.

\paragraph{The static optimal control problem $(P_s)$.}
We consider the nonlinear constrained minimization problem
$$(P_s)\qquad\qquad\inf_{u\in U}\;\;J_s(u)= f^0(y,u),$$
where $y\in X$ is the corresponding weak  solution\footnote{Recall that the form of weak solution $y\in X$ is
$$
\langle y,A^*\varphi\rangle_X+\langle f(y,u),\varphi\rangle_X=0,\;\;\;\;\text{for any}\;\;\varphi\in D(A^*).
$$
} of 
$$Ay+f(y,u)=0.$$
Likewise, we  assume that the problem $(P_s)$ has at least one optimal solution (sufficient conditions ensuring existence are standard, see, e.g., \cite{ItoKunisch}). Let $(y_s,u_s)\in X\times U$ be any of them. We assume  that $(y_s,u_s)$ has a normal 
extremal lift. For instance, this normality does 
occur when $A$ is a uniformly elliptic operator and the nonlinearity $f$ satisfies a monotone condition with respect to the state variable (see e.g. \cite[Chapter 5, Theorem 1.2]{LiXunjing}).
According to the Lagrange multiplier rule (see \cite{ItoKunisch,LiXunjing,T}), there exists $\lambda_s\in X$ such that
\begin{equation}\label{721-5}
\left\{\begin{split}
Ay_s+H_\lambda(y_s,\lambda_s,u_s)=0,\\
-A^*\lambda_s-H_y(y_s,\lambda_s,u_s)=0\\
\end{split}\right.
\end{equation}
and
\begin{equation}\label{721-3}
H_u(y_s,\lambda_s,u_s)=0.
\end{equation}
Here, $H$ is the Hamiltonian function defined by \eqref{def_H}.

Note that $(y_s,\lambda_s,u_s)$ is an equilibrium point of the differential system \eqref{721-1}, satisfying the constraint \eqref{721-2}.
This remark is crucial in order to understand the turnpike property. Indeed, we are going to prove that, under appropriate assumptions, the equilibrium point $(y_s,\lambda_s,u_s)$ is hyperbolic, in the sense that, if we linearize the system of equations \eqref{721-1} around the point $(y_s,\lambda_s,u_s)$, then we obtain a linear system that is hyperbolic. This feature, adequately interpreted, implies the exponential turnpike property, locally around $(y_s,\lambda_s,u_s)$.

\paragraph{The exponential turnpike property.}
In what follows, we assume that the linear bounded operator $H_{uu}(y_s,\lambda_s,u_s)$ on $X$ is negative definite and boundedly invertible\footnote{This assumption is standard in optimal control theory, and it is usually referred  as a {\it strong Legendre condition}, see, e.g., \cite{Trelatbook}. It implies that the optimal control can be locally solved by the maximum condition in terms of the optimal state and adjoint state.}, and that the operator $H_{yu}H_{uu}^{-1}H_{uy}-H_{yy}$, taken at $(y_s,\lambda_s,u_s)$, is nonnegative.
At the point $ (y_s,\lambda_s,u_s)$, we define
\begin{equation}\label{727-10}
\mathcal{A}= A+H_{\lambda y}-H_{\lambda u}H_{uu}^{-1}H_{uy} ,
\end{equation}
and 
\begin{equation}\label{727-11}
\mathcal{C}^*\mathcal {C}= H_{yu}H_{uu}^{-1}H_{uy}-H_{yy} ,
\end{equation}
for some $\mathcal C\in L(X,X)$, where 
$$
H_{\lambda y}=H_{\lambda y}(y_s,\lambda_s,u_s),\;\;H_{\lambda u}=H_{\lambda u}(y_s,\lambda_s,u_s),\;\; H_{uu}=H_{uu}(y_s,\lambda_s,u_s) ,$$
and 
$$
H_{u y}=H_{uy}(y_s,\lambda_s,u_s)
,\;\;H_{y u}=H_{yu}(y_s,\lambda_s,u_s),\;\; H_{yy}=H_{yy}(y_s,\lambda_s,u_s).$$

\begin{theorem}\label{731-10}
Assume that the pair $(\mathcal A,H_{\lambda u})$ is exponentially stabilizable\footnote{The pair $(\mathcal A,H_{\lambda u})$ is said to be exponentially stabilizable if and only if there exists an operator $\mathcal K\in L(X, U)$ such that the operator  $\mathcal A+H_{\lambda u}\mathcal K$ is exponentially stable, i.e., the operator $\mathcal A+H_{\lambda u}\mathcal K$ generates a $C_0$ semigroup $(\mathcal S(t))_{t\geq0}$ satisfying $\|\mathcal S(t)\|_{L(X,X)}\leq ce^{-\nu t}$ for all $t\geq0$, for some $c>0$ and $\nu>0$.} and that the pair $(\mathcal A,\mathcal C)$ is exponentially detectable\footnote{The pair $(\mathcal A,\mathcal C)$ is said to be exponentially detectable if $(\mathcal A^*,\mathcal C^*)$ is exponentially stabilizable.}.
Then, there exist positive constants $\varepsilon$,  $\mu$ and $c$  such that for any $T>0$, if
\begin{equation}\label{727-88}
\|y_0-y_s\|_X+\|\lambda_s\|_X\leq \varepsilon,
\end{equation}
any optimal extremal triple  $(y^T(\cdot),u^T(\cdot),\lambda^T(\cdot))$ of $(OCP^T)$ has the exponential turnpike property
\begin{equation} \label{727-18}
\left\Vert y^T(t)-y_s\right\Vert_X+\left\Vert u^T(t)-u_s\right\Vert_U+\left\Vert \lambda^T(t)-\lambda_s\right\Vert_X\leq c \left( e^{-\mu t}+e^{-\mu(T-t)} \right) ,
\end{equation}
for almost every $t\in [0,T]$.
\end{theorem}

\begin{remark}
The above theorem extends the result established in  \cite{TZ1} for general finite-dimensional optimal control problems. It is as well local, requiring the smallness assumption \eqref{727-88}. Ruling out this assumption would require to have a knowledge of global properties of the dynamics.
\end{remark}

\begin{remark}
When the control system \eqref{731-9} is a controlled semilinear heat equation, the exponential turnpike property has been established in \cite{PZ2} for one only of the solutions of the  optimality system \eqref{721-1}-\eqref{721-2} under some smallness conditions. Note that Theorem~\ref{731-10} gives the conclusion for all optimal extremal solutions under certain smallness conditions.
\end{remark}

\begin{remark}
Note that we assume that  $f$ and $f^0$ are $C^2$-smooth and 
globally Lipschitz with respect to the state variable. Under such a globally Lipschitz condition, we would get the existence and uniqueness of solutions for a given Cauchy problem.   To ensure that we can obtain the linearized  system
of the optimality systems resulting from the Pontryagin maximum principle for $(OCP^T)$,  the $C^2$-regularity of the dynamic seems to be necessary. In practice, it is often the case that, even though globally Lipschitz properties are not satisfied, we can however reduce the problem to the globally Lipschitz situation. Indeed, if we know in advance that solutions under consideration remain in a bounded set, and if the dynamics and their derivatives are bounded on bounded sets, then one can change $f$ and $f^0$ by multiplying them by zero at infinity, so that they are smooth and of compact support, with this compact support containing all solutions of interest. With such a reasoning, we reduce the problem to dynamics that are globally Lipschitz.

Here, however, it is just required to assume that $(OCP^T)$ exists at least one solution, having a normal extremal lift. 
Then, similar things can be also done in each particular instance,  such as the cubic semilinear heat equation in Example~\ref{nonlinearex} below.

\end{remark}

\begin{example}\label{nonlinearex}
Let $\Omega\subset\mathbb{R}^3$ be an open and bounded domain 
with a $C^2$ boundary, and let $\omega\subset\Omega$ be a nonempty open subset. Denote by $\chi_\omega$ the characteristic function of the subset $\omega$. 
Given $T>0$, $y_d\in L^2(\Omega)$ and $y_0\in L^2(\Omega)$, we consider the optimal control problem 
\begin{equation*}
\text{Minimize}\;\;\; \frac{1}{2}\int_0^T\int_{\Omega}
|y(x,t)-y_d(x)|^2\,dx\,dt+\frac{1}{2}\int_0^T\int_{\omega}|u(x,t)|^2\,dx\,dt,
\end{equation*}
subject to $(y,u)\in C([0,T];L^2(\Omega))\times L^2(0,T;L^2(\Omega))$ satisfying the semilinear heat
equation with a cubic nonlinearity
\begin{equation*}\left\{
\begin{split}
&y_{t}-\triangle y+y^3=\chi_{\omega}u\;\;\;\;\text{in}\;\;\Omega\times(0,T),\\
&y=0\;\;\;\;\text{on}\;\;\partial \Omega\times(0,T),\\
&y(0)=y_0 \;\;\;\;\text{in}\;\;\Omega.
\end{split}\right.
\end{equation*}
This semilinear heat equation is well-posed. More precisely, given $y_0\in L^2(\Omega)$ and $u\in L^2(\omega\times(0,T))$, there exists a unique solution $y\in C([0,T];L^2(\Omega))\cap L^2(0,T;H^1_0(\Omega))$.  
Moreover, for each $T>0$,
there exists at least one optimal solution $(y^T(\cdot),u^T(\cdot))$. Meanwhile, there is an adjoint sate $\lambda^T(\cdot)$ 
such that $(y^T(\cdot),\lambda^T(\cdot))$ satisfies the optimality systems (cf., e.g., \cite[Section 3.1]{PZ2})
\begin{equation*}\left\{
\begin{split}
&y^T_{t}-\triangle y^T+(y^T)^3=\chi_{\omega}\lambda^T\;\;\;\;\text{in}\;\;\Omega\times(0,T),\\
&y^T=0\;\;\;\;\text{on}\;\;\partial \Omega\times(0,T),\\
&y^T(0)=y_0 \;\;\;\;\text{in}\;\;\Omega,\\
& \lambda_t^T+\Delta\lambda^T-3(y^T)^2\lambda^T=y^T-y_d\;\;\;\;\text{in}\;\;\Omega\times(0,T),\\
&\lambda^T=0\;\;\;\;\text{on}\;\;\partial \Omega\times(0,T),\\
&\lambda^T(T)=0 \;\;\;\;\text{in}\;\;\Omega,
\end{split}\right.
\end{equation*}
and 
$$u^T(t)=\chi_\omega\lambda^T(t)\;\;\;\;\text{for almost every}\;\; t\in(0,T).$$

The corresponding static optimal control problem is 
\begin{equation*}\text{Minimize}\;\;\; \frac{1}{2}\int_{\Omega}
|y(x)-y_d(x)|^2\,dx+\frac{1}{2}\int_{\omega}|u(x)|^2\,dx,
\end{equation*}
subject to $(y,u)\in L^2(\Omega)\times L^2(\Omega)$ satisfying 
\begin{equation*}\left\{\begin{split}
&-\triangle y+y^3=\chi_\omega u\;\;\text{in}\;\;\Omega,\\
& y=0\;\; \text{on}\;\;\partial\Omega.
\end{split}\right.\end{equation*}
Obviously, any minimizer $(y_s,u_s)$ of this minimization problem satisfies 
$$
\|y_s-y_d\|^2_{L^2(\Omega)}+\| u_s\|^2_{L^2(\omega)}\leq \|y_d\|_{L^2(\Omega)}^2.
$$ 
Moreover, there exists $\lambda_s\in L^2(\Omega)$ such that 
\begin{equation*}\left\{
\begin{split}
&-\triangle \lambda_s+3y_s^2\lambda_s+y_s-y_d=0 \;\;\text{in}\;\;\Omega,\\
& \lambda_s=0\;\; \text{on}\;\;\partial\Omega.
\end{split}\right.\end{equation*}
In view of applying the Sobolev imbedding inequality and the elliptic regularity theory, we note that there exists $c=c(\Omega)>0$ such that 
\begin{equation*}
\|{y}_s\|_{H_0^1(\Omega)}+\|{\lambda}_s\|_{H^1_0(\Omega)}\leq c\|y_d\|_{L^2(\Omega)}.
\end{equation*}
Hence, there exists $\varepsilon >0$ such that the condition \eqref{727-88} in Theorem~\ref{731-10} holds whenever
$y_d$ and $y_0$ have small enough $L^2$-norms (see also more details in \cite[Section 3.2]{PZ2}).

We apply Theorem~\ref{731-10} with $X= L^2(\Omega)$, $U=L^2(\omega)$, $A=\triangle$ defined on the domain $D(A)=H^2(\Omega)\cap H_0^1(\Omega)$, $f(y,u)=-y^3+\chi_\omega u$, and $f^0(y,u)=\frac{1}{2}\|y-y_d\|_{L^2(\Omega)}^2+\frac{1}{2}\|u\|_{L^2(\omega)}^2$.
Note that
$$
\mathcal A=A+O(\|y_d\|_{L^2(\Omega)}),\;\; H_{\lambda u}=\chi_\omega I+O(\|y_d\|_{L^2(\Omega)}),\;\;\mathcal C= I+O(\|y_d\|_{L^2(\Omega)}),
$$
where $I$ is the identity operator on $L^2(\Omega)$.
Since the semigroup generated by $A$ is exponentially stable,  by perturbation, the pairs $(\mathcal A,H_{\lambda u})$ and $(\mathcal A^*,\mathcal C^*)$ are also exponentially stabilizable whenever the $L^2$-norm of $y_d$ is sufficiently small (see, e.g., \cite[Chapter 3, Theorem 1.1]{Pa})).   
Therefore, the exponential turnpike property is satisfied provided $\Vert y_d\Vert_{L^2}$ and  $\Vert y_0\Vert_{L^2}$ are small enough.
\end{example}

\medskip

\subsection{Periodic exponential turnpike for linear quadratic problems with periodic tracking trajectory}\label{sec_periodic}
Let $X$, $U$ and $V$ be Hilbert spaces identified with their respective duals.
As in the previous section, we first define the dynamical optimal control problem $(LQ^T)$, formulated as a linear-quadratic optimal control problem with a periodic tracking trajectory.
Since the cost functional depends on $t$ in a periodic way, we replace the static optimal control problem with a periodic optimal control problem $(LQ^\Pi)$, whose solution yields the referred turnpike. The exponential turnpike result 
is then stated in Theorem \ref{periodictarget}.

\paragraph{The dynamical optimal control problem $(LQ^T)$.}

Given any $y_0\in X$, we consider the linear control system
\begin{equation}\label{61720}
\left\{
\begin{split}
&\dot y(t) = Ay(t) + Bu(t),\;\;t>0,\\
&y(0)=y_0,
\end{split}
\right.
\end{equation}
where the operator $A:D(A)\subset X\rightarrow X$  generates a $C_0$ semigroup in $X$, and $B\in L(U,X)$ is a linear bounded control operator.
 
Let $\Pi>0$ be a fixed positive real number.  Let $y_d(\cdot)\in C([0,+\infty);X)$ and
$u_d(\cdot)\in L^2_{loc}(0,+\infty;U)$ be two $\Pi$-periodic functions such that
$$
y_d(t+\Pi)=y_d(t),\;\;u_d(t+\Pi)=u_d(t),\;\;\;\;\text{for a.e.}\;\;\;t>0.
$$ 
Let $C\in L(X,V)$ be a linear bounded observation operator, and let $Q\in L(X,X)$ be an invertible positive definite operator.
For any $T>0$, we consider the optimal control problem
\begin{equation*}\label{xiao-1}
(LQ^T)\qquad\qquad\inf_{u(\cdot)\in L^2(0,T;U)}J^T(u(\cdot))= 
\frac{1}{2}\int_0^T \left( \|C(y(t)-y_d(t))\|_V^2 + \Vert Q^{1/2}(u(t)-u_d(t))\Vert_U^2\right) dt,
\end{equation*}
where $y(\cdot)\in C([0,T];X)$ is the solution of \eqref{61720} with the control $u(\cdot)$.  
In the literature, this minimization problem is usually referred to as a linear quadratic optimal problem with a periodic tracking trajectory.  

The problem $(LQ^T)$ has a unique optimal solution $(y^T(\cdot),u^T(\cdot))$.
Moreover, following \cite{ItoKunisch,Lions} or \cite[Chapter 4, Theorem 1.6]{LiXunjing}),  there exists $\lambda^T(\cdot)\in C([0,T];X)$
such that 
\begin{equation}\label{j152}
\left\{
\begin{split}
\dot{y}^T(t) &= Ay^T(t)+BQ^{-1}B^*\lambda^T(t)+Bu_d(t),\qquad y^T(0)=y_0, \\
\dot{\lambda}^T(t) &= C^*Cy^T(t)-A^*\lambda^T(t)-C^*Cy_d(t),\qquad \lambda^T(T)=0,
\end{split}
\right.
\end{equation}
in the mild sense along $[0,T]$ and
\begin{equation*}\label{j151}
u^T(t)=u_d(t)+Q^{-1}B^*\lambda^T(t), \;\;\text{a.e.}\;\;t\in[0,T].
\end{equation*}

\paragraph{The periodic optimal control problem $(LQ^\Pi)$.}
In the present case where the tracking terms in the cost functional depend on $t$, the turnpike property cannot anymore be captured by a corresponding static optimal control problem.
Instead, we consider the periodic optimal control problem
\begin{equation*}\label{xiao-2}
(LQ^\Pi)\qquad\qquad\text{inf}\;\; J^\Pi(y(\cdot),u(\cdot))= 
\frac{1}{2}\int_0^\Pi \left( \|C(y(t)-y_d(t))\|_V^2
+\Vert Q^{1/2}(u(t)-u_d(t))\Vert_U^2\right) dt,
\end{equation*}
where $(y(\cdot),u(\cdot))\in C([0,\Pi];X)\times L^2(0,\Pi;U)$ is a mild solution of
\begin{equation*}
\left\{
\begin{split}
&\dot{y}(t)=Ay(t)+Bu(t),\;\;t\in[0,\Pi],\\
&y(0)=y(\Pi).\\
\end{split}\right.
\end{equation*}
Existence and uniqueness for such periodic optimal control problems, as well as first-order necessary conditions for optimality, have been widely studied in the existing literature (see, for instance, \cite{Barbu,EG,TWZ} or \cite[Chapter 4, Proposition 5.2]{LiXunjing} and references therein).
Since the problem $(LQ^\Pi)$ is convex, it is well known that $(y^\Pi(\cdot),u^\Pi(\cdot))$ is an optimal pair for $(LQ^\Pi)$ if and only if there exists an adjoint state $\lambda^\Pi\in C([0,\Pi];X)$ such that
\begin{equation}\label{xiaoextremalsystLQ}
\left\{
\begin{split}
\dot{y}^\Pi(t) &= Ay^\Pi(t)+BQ^{-1}B^*\lambda^\Pi(t)+Bu_d(t),\qquad y^\Pi(0)=y^\Pi(\Pi), \\
\dot{\lambda}^\Pi(t) &= C^*Cy^\Pi(t)-A^*\lambda^\Pi(t)-C^*Cy_d(t),\qquad \lambda^\Pi(0)=\lambda^\Pi(\Pi),
\end{split}
\right.
\end{equation}
in the mild sense along $[0,\Pi]$, and
\begin{equation}\label{xiao61718}
u^\Pi(t)=u_d(t)+Q^{-1}B^*\lambda^\Pi(t),\;\;\text{a.e.}\;\;t\in[0,\Pi].
\end{equation}

\paragraph{The periodic exponential turnpike property.}

\begin{theorem}\label{periodictarget}
Assume that the pair $(A,B)$ is exponentially stabilizable and that the pair $(A,C)$ is exponentially detectable. Then:
\begin{itemize}
\item The problem $(LQ^\Pi)$ has a unique solution $(y^\Pi(\cdot),u^\Pi(\cdot))$, which has a unique extremal lift $(y^\Pi(\cdot),u^\Pi(\cdot),\lambda^\Pi(\cdot))$ solution of \eqref{xiao61718} and \eqref{xiaoextremalsystLQ}. We extend it by $\Pi$-periodicity over $[0,+\infty)$.

\item There exist positive constants $c$ and $\nu$ such that, for any $T > 0$,
\begin{equation}\label{10271}
\left\Vert y^T(t)-y^\Pi(t)\right\Vert_X+\left\Vert u^T(t)-u^\Pi(t)\right\Vert_U+\left\Vert \lambda^T(t)-\lambda^\Pi(t)\right\Vert_X
\leq c\left( e^{-\nu t}+e^{-\nu (T-t)}\right) ,
\end{equation}
for almost every $t\in[0,T]$.
\end{itemize}
\end{theorem}

\begin{remark}\label{explicit}
From the proof of the theorem, we infer the following explicit formulas in order to compute the optimal triple $(y^\Pi(\cdot),u^\Pi(\cdot),\lambda^\Pi(\cdot))$.
We claim that
$$
y^\Pi(t)=z(t)- E q(t),\quad
\lambda^\Pi(t)=-Pz(t)+(I+PE)q(t),\quad
u^\Pi(t)=u_d(t)+Q^{-1}B^*\lambda^\Pi(t),
$$
for almost every $t\in[0,\Pi]$, where:
\begin{itemize}
\item $P\in L(X,X)$ is the unique nonnegative definite self-adjoint operator solution of the operator algebraic Riccati equation
$$A^*P+PA-PBQ^{-1}B^*P+C^*C=0,$$
or equivalently,
\begin{equation*}
2\langle PAx,x\rangle_X-\langle PBQ^{-1}B^*Px, x\rangle_X+\langle Cx, Cx\rangle_V=0, \;\;\;\forall x\in D(A);
\end{equation*}
\item $E\in L(X,X)$ is defined by
$$E= -\int_0^{+\infty} S(t)BQ^{-1}B^*S(t)^*\,dt ,$$
where $(S(t))_{t\geq 0}$ is the (exponentially stable) $C_0$ semigroup generated by the 
operator $A-BQ^{-1}B^*P$;
\item $z(t)$ and $q(t)$ are the $\Pi$-periodic trajectories defined by
\begin{multline*}
z(t)=S(t)(I-S(\Pi)\big)^{-1}\int_0^\Pi S(\Pi-\tau)\big((I+ E P)Bu_d(\tau)- EC^*Cy_d(\tau)\big)\,d\tau
\\+\int_0^tS(t-\tau)\big((I+E P)Bu_d(\tau)-EC^*Cy_d(\tau)\big)\,d\tau,
\end{multline*}
and
\begin{multline*}
q(t)=S(\Pi-t)^*(I-S(\Pi)^*\big)^{-1}\int_0^\Pi S(\Pi-\tau)^*\big(-PBu_d(\Pi-\tau)+C^*Cy_d(\Pi-\tau))\big)\,d\tau
\\+\int_0^{\Pi-t}S(\Pi-t-\tau)^*\big(-PBu_d(\Pi-\tau)+C^*Cy_d(\Pi-\tau)\big)\,d\tau,
\end{multline*}
for every $t \in[0,\Pi]$.
\end{itemize}
These facts are proved in Section \ref{1252} (more specifically, see Lemma \ref{lemma_explicit}), as well as Theorem \ref{periodictarget}.
\end{remark}

\begin{remark}
As it can be seen from the proof of the theorem, the best exponential decay constant $\nu$ in \eqref{10271} can be characterized as the exponential stability rate for a $C_0$ semigroup resulting from the operator algebraic Riccati equation. 
\end{remark}

It follows from Theorem~\ref{periodictarget} that there exists $\eta>0$ such that, for any initial condition $y_0$, the optimal triple $(y^T(\cdot),u^T(\cdot),\lambda^T(\cdot))$solution of $(LQ^T)$ is exponentially close to the periodic optimal triple $(y^\Pi(\cdot),u^\Pi(\cdot),\lambda^\Pi(\cdot))$ solution of $(LQ^\Pi)$ over the middle time interval $[\eta, T-\eta]$ whenever $T$ is large enough. 
Boundary layers may occur at $t=0$ and $t=T$ for the optimality system, and the exponential closeness is observed in the middle piece of optimal trajectories.
This result means that, except at the extremities of the time frame, the optimal trajectory, as well as the optimal control and the associated adjoint state, is almost $\Pi$-periodic.
It is worth mentioning that similar results have been discussed in \cite{AL}, \cite{Samuelson1976} and \cite[Chapter 6]{Z2} for some finite-dimensional optimal control problems.

\begin{example}
Let $\Omega\subset\mathbb{R}^n$ ($n\geq1$) be an open and bounded domain 
with a $C^2$ boundary, and let $\omega_i\subset\Omega$, $i=1,2$, be nonempty open subsets. Denote by $\chi_{\omega_i}$, $i=1,2$, the associated characteristic function.
Let $y_d\in C([0,+\infty);L^2(\Omega))$ be a periodic tracking trajectory, satisfying $y_d(t,\cdot)=y_d(t+1,\cdot)$ for any $t\geq 0$.
Given any $T>0$, we consider the optimal control problem
\begin{equation*}\text{Minimize}\;\;\; \frac{1}{2}\int_0^T\int_{\omega_1}
|y(x,t)-y_d(x,t)|^2\,dx\,dt+\frac{1}{2}\int_0^T\int_{\omega_2}|u(x,t)|^2\,dx\,dt,
\end{equation*}
over all possible $(y,u)\in C([0,T];L^2(\Omega))\times L^2(0,T;L^2(\Omega))$ satisfying (in the mild sense)
\begin{equation*}\left\{
\begin{split}
&y_{t}-\triangle y +a(x)y=\chi_{\omega_2}u\;\;\;\;\text{in}\;\;\Omega\times(0,T),\\
&y=0\;\;\;\;\text{on}\;\;\partial \Omega\times(0,T),\\
&y(0)=y_0\;\;\;\;\text{in}\;\;\Omega,
\end{split}\right.
\end{equation*}
where $a(\cdot)\in L^\infty(\Omega)$ and $y_0\in  L^2(\Omega)$. 
We apply Theorem~\ref{periodictarget} with $X=U=V=L^2(\Omega)$, $A=\triangle-a(\cdot)I$ defined on the domain $D(A)=H^2(\Omega)\cap H_0^1(\Omega)$, $B=\chi_{\omega_2}I$, and $C=\chi_{\omega_1}I$.
Here, $I$ is the identity operator on $L^2(\Omega)$. Since the above heat equation with distributed control localized in $\omega_i$, $i=1,2$, is null controllable in any finite time (see for instance \cite{Zcon}), the pairs $(A,B)$ and $(A^*,C^*)$ are exponentially stabilizable. Then, according to Theorem~\ref{periodictarget}, this optimal control problem has the periodic  exponential turnpike property.
\end{example}

\begin{example}
Let $\Omega\subset\mathbb{R}^n$ ($n\geq1$) be an open and bounded domain 
with a $C^2$ boundary, and let $\omega_i\subset\Omega$, $i=1,2$, be nonempty open subsets. 
Let $z_d\in C([0,+\infty);L^2(\Omega))$ be a periodic tracking trajectory, satisfying $z_d(t,\cdot)=z_d(t+1,\cdot)$ for any $t\geq 0$.
Given any $T>0$, we consider the optimal control problem
\begin{equation*}\text{Minimize}\;\;\; \frac{1}{2}\int_0^T\int_{\omega_1}
|z_t(x,t)-z_d(x,t)|^2\,dx\,dt+\frac{1}{2}\int_0^T\int_{\omega_2}|u(x,t)|^2\,dx\,dt,
\end{equation*}
over all possible $(z,u)\in C([0,T];H_0^1(\Omega))\cap C^1([0,T]; L^2(\Omega))\times L^2(0,T;L^2(\Omega))$ satisfying (in the mild sense)
\begin{equation*}\left\{
\begin{split}
&z_{tt}-\triangle z=\chi_{\omega_2}u\;\;\;\;\text{in}\;\;\Omega\times(0,T),\\
&z=0\;\;\;\;\text{on}\;\;\partial \Omega\times(0,T),\\
&z(0)=z_0,\;\;z_t(0)=z_1\;\;\;\;\text{in}\;\;\Omega,
\end{split}\right.
\end{equation*}
where  $z_0\in H_0^1(\Omega)$ and $z_1\in L^2(\Omega)$. Writing the wave equation as a first-order system, we apply Theorem~\ref{periodictarget} with $X=H_0^1(\Omega)\times L^2(\Omega)$, $U=V=L^2(\Omega)$,
\begin{equation*}
A=
\begin{pmatrix}
0 &I\\
\triangle & 0\\
\end{pmatrix}\;\;\; \text{defined on the domain}\;\;D(A)=\big(H^2(\Omega)\cap H_0^1(\Omega)\big)\times H_0^1(\Omega),
\end{equation*}
\begin{equation*}
B=\begin{pmatrix}
0 \\
\chi_{\omega_2}I\\
\end{pmatrix},\;\;\; 
C=\begin{pmatrix}
0 &
\chi_{\omega_1}I\\
\end{pmatrix},
\end{equation*}
where $I$ the identity operator on $L^2(\Omega)$. Note that $A^*=-A$.
It is well known (see \cite{BardosLebeauRauch}) that, under the assumption that 
$(\Omega,\omega_i)$, $i=1,2$, satisfies the so-called \textit{Geometric Control Condition}\footnote{The Geometric Control Condition stipulates that every ray of geometrical optics that propagates in $\Omega$ and  reflects on its boundary should intersect $\omega_i$ within time  $T$.}, there exists $T_0>0$ such that the wave equation with the distributed control localized in $\omega_i$, $i=1,2$, is null controllable in any time $T>T_0$. Therefore,
the pairs $(A,B)$ and $(A^*,C^*)$ are exponentially stabilizable. Then, according to Theorem~\ref{periodictarget}, this optimal control problem has the periodic exponential turnpike property if $T>T_0$.
\end{example}

\subsection{Particular case: tracking a point}
Consider the linear quadratic optimal control problem $(LQ^T)$ of Section \ref{sec_periodic}.
If $y_d(t)=y_d\in X$ and $u_d(t)=u_d\in U$ do not depend on $t$, then Theorem~\ref{731-10} (Section~\ref{sub2}) can be applied and the referred turnpike is an optimal solution of the static optimal control problem 
\begin{equation*}
(P_s)\qquad\qquad\inf J_s(y,u)=
\frac{1}{2}\left(\|C(y-y_d)\|_V^2 + \Vert Q^{1/2}(u-u_d)\Vert_U\right),
\end{equation*}
over the set of all $(y,u)\in X\times U$  satisfying the constraint $Ay+Bu=0$.
Since $(P_s)$ is a convex programming problem, it is well known that 
 $(y_s,u_s)\in X\times U$ is an optimal solution of $(P_s)$ if and only if there exists an adjoint state $\lambda_s\in X$ (see, e.g., \cite{ItoKunisch,Lions} or \cite[Chapter 6]{T}) such that $u_s = u_d + Q^{-1}B^*\lambda_s$ and
\begin{equation}\label{6142}
\left\{\begin{split}
&Ay_s + BQ^{-1}B^*\lambda_s + Bu_d = 0,\\
&C^*C y_s - A^*\lambda_s -C^*Cy_d = 0.
\end{split}\right.
\end{equation}
Besides, it follows from \cite{TZZ1} that the optimal solution of the periodic optimal problem $(LQ^\Pi)$ coincides with that of the corresponding steady-state optimal control problem $(P_s)$ (i.e., $(y^\Pi(\cdot),u^\Pi(\cdot))\equiv (y_s,u_s)$). More precisely, we have the following result.

\begin{theorem}\label{thm1LQ}
Assume that $(A,B)$ is exponentially stabilizable and $(A,C)$
is exponentially detectable. If $y_d(\cdot)\equiv y_d\in X$ and $u_d(\cdot)\equiv u_d\in U$, then there exist positive constants $c$ and $\nu$ such that, for any $T>0$,
\begin{multline}\label{6175}
\left\Vert y^T(t)- y_s\right\Vert_X+\left\Vert \lambda^T(t)-\lambda_s\right\Vert_X+\left\Vert u^T(t)-u_s\right\Vert_U \\
\leq c \left( \|y_0-y_s\|_X+\|\lambda_s\|_X \right) \left( e^{-\nu t}+e^{-\nu (T-t)} \right) .
\end{multline}
\end{theorem}

This result means that, in the linear quadratic framework, Theorem~\ref{731-10} holds true \emph{globally}. It unifies and extends
the exponential turnpike theorems established in \cite{PZ1,TZ1}.

\subsection{A numerical simulation}
In this section, we provide a simple example in order to numerically illustrate the periodic turnpike phenomenon in finite-dimensional case. 
Given any $T>0$, we consider the optimal control problem of minimizing the cost functional
\begin{equation*}
\frac{1}{2}\int_0^T \big((x(t)-\cos (2\pi t))^2+(y(t)-\sin (2\pi t))^2+u(t)^2\big)\,dt.
\end{equation*}
for the two-dimensional control system
\begin{equation*}
\dot x(t)=y(t),\qquad \dot y(t)=u(t), \;\;\;\;t\in(0,T),
\end{equation*}
with fixed initial condition $(x(0),y(0))=(0.1,0)$. 
Here, the target trajectories are $1$-periodic.
The explicit formulas for the expected periodic turnpike are given in Remark~\ref{explicit} and have been numerically computed with \texttt{Matlab}.
More precisely, to fit in the framework that has been developed previously, we set
$$
A = \begin{pmatrix} 0 & 1 \\ 0 & 0 \end{pmatrix},\qquad
B = \begin{pmatrix} 0 \\ 1 \end{pmatrix},\qquad C=Q=I,\qquad 
u_d\equiv 0,\qquad y_d(t) = \begin{pmatrix} \cos (2\pi t) \\ \sin (2\pi t)\end{pmatrix},
$$
and then, using \texttt{Matlab}:
\begin{itemize}
\item We solve the Riccati equation $A^*P+PA-PBB^*P+I=0$.
\item We solve the Lyapunov equation $(A-BB^*P)E+E(A-BB^*P)^*-BB^*=0$.
\item We set $\mathcal{T} = \begin{pmatrix} I & -E \\ -P & I+PE \end{pmatrix}$ and $S(t)=\exp(t(A-BB^*P))$.
\item We compute
$$
z(t) = -S(t)(I-S(1))^{-1}\int_0^1 S(1-\tau)Ey_d(\tau)\,d\tau - \int_0^t S(t-\tau)Ey_d(\tau)\,d\tau
$$
and
$$
q(t) = S(1-t)^*(I-S(1)^*)^{-1}\int_0^1 S(1-\tau)^*y_d(1-\tau)\, d\tau + \int_0^{1-t} S(1-t-\tau)^*y_d(1-\tau)\, d\tau,
$$
for $t\in[0,1]$. This can be done by noting that $z(\cdot)$ and $q(\cdot)$ are solutions of some ordinary differential equations and by using numerical integration.
\item Then, the reference turnpike trajectory and adjoint state are given by
$$
\begin{pmatrix}\big(\bar x(t), \bar y(t)\big)\\\big( \bar \lambda_x(t),\bar \lambda_y(t)\big)\end{pmatrix} = \mathcal{T} \begin{pmatrix} z(t)\\ q(t)\end{pmatrix},\qquad t\in[0,1].
$$
\end{itemize}
The optimal extremal $(x(\cdot),y(\cdot),\lambda_x(\cdot),\lambda_y(\cdot),u(\cdot))$, solution of the first-order optimality system derived from the Pontryagin maximum principle, has been computed in time $T = 20$ by using a direct method of numerical optimal control (see \cite{Trelatbook}), more precisely, we have discretized the above optimal control problem using a Crank-Nicolson method and we have then used the automatic differentiation code \texttt{AMPL} (see \cite{AMPL}) combined with the optimization routine \texttt{IpOpt} (see \cite{IPOPT}) on a standard desktop machine.

\begin{figure}[h]
\centering 
\includegraphics[width=15cm]{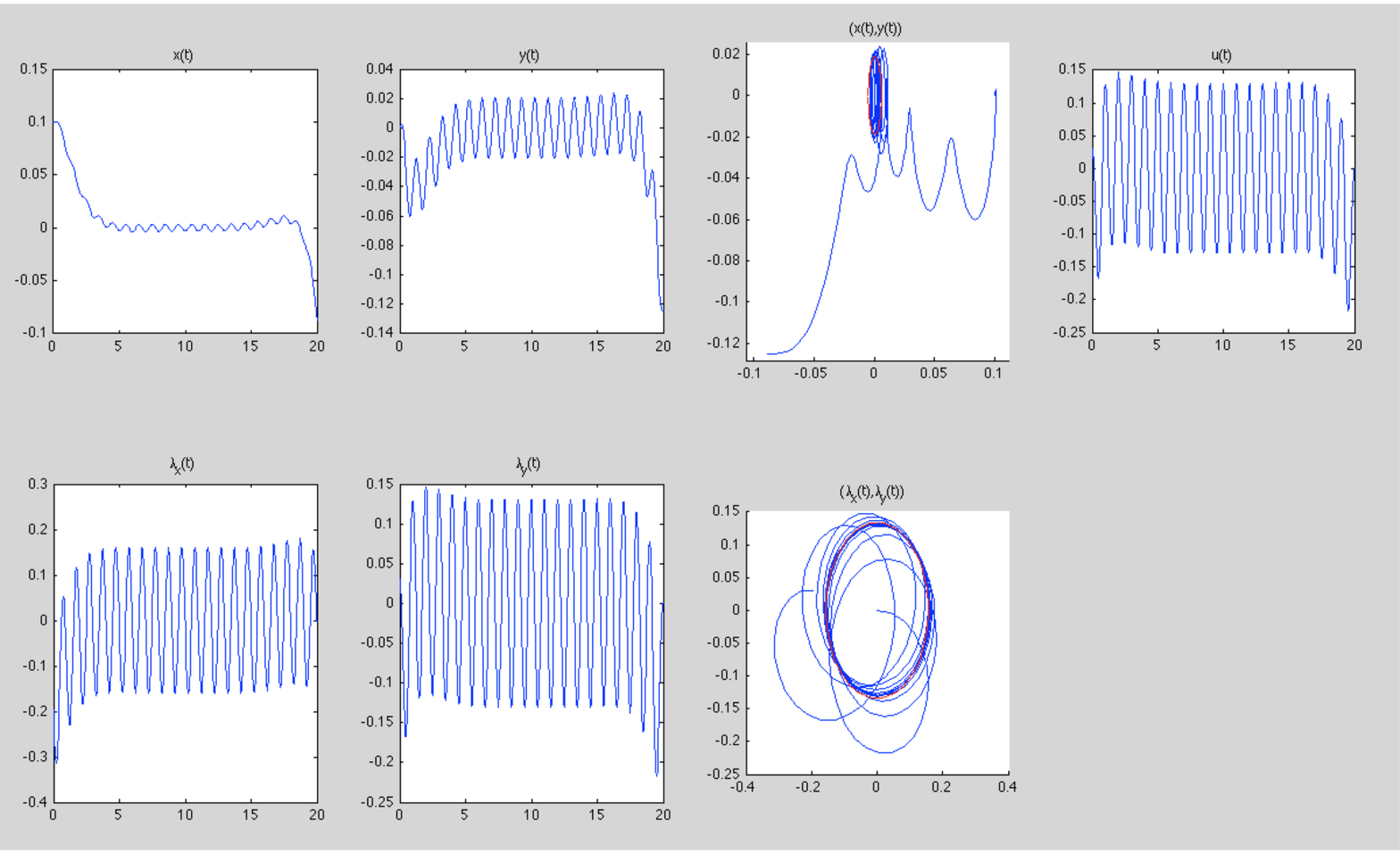}
\caption{Example of a periodic turnpike} \label{fig1}
\end{figure}

The turnpike property can be observed on Figure \ref{fig1}. 
As expected, except transient initial and final arcs, the extremal  $(x(\cdot), y(\cdot),\lambda_x(\cdot),\lambda_y(\cdot))$ (in blue) remains close to the periodic turnpike
$(\bar x(\cdot), \bar y(\cdot),\bar \lambda_x(\cdot),\bar \lambda_y(\cdot))$ (in red).

\subsection{Conclusion and further comments}\label{con}
We have established the exponential turnpike property around an optimal steady-state for general nonlinear infinite-dimensional optimal control problems under certain stability and smallness assumptions. We have then established the periodic exponential turnpike property 
for linear quadratic optimal control problems with periodic tracking trajectories, for which the turnpike is an optimal solution to a periodic optimal control problem. To the best of our knowledge, the latter result is new even in finite dimension.
Moreover, the optimal periodic solution has been explicitly characterized by means of a dichotomy transformation on the solutions of the operator algebraic Riccati and Lyapunov equations.
 
Some possible perspectives are in order.

\begin{enumerate}

\item We believe that the periodic turnpike phenomenon appears in many concrete situations, and is for instance reminiscent of many biological processes (see, e.g., \cite{GAB}). Then, obtaining a nonlinear version of our periodic turnpike theorem would certainly model many possible problems in life sciences.
When linearizing the extremal system (derived from the Pontryagin maximum principle) around the optimal periodic trajectory, we obtain a linear time-periodic Hamiltonian system instead of an autonomous one. Identifying adequate hyperbolic properties is then an open problem. In a forthcoming paper, we will investigate the periodic turnpike phenomenon for optimal controls of general nonlinear time-periodic systems.

\item Our analysis was restricted to bounded control and observation operators. For instance, we were able to treat heat and wave equations with \emph{internal} control, but not with a \emph{boundary} control because then the control operator is unbounded. However, it is expected that similar exponential turnpike properties may also be established in a general context for infinite-dimensional optimal control systems with unbounded input and output operators. We refer to \cite{GTZ} for a result in this direction, for a one-dimensional wave equation with Neumann boundary control.

\item
Although the results established in this paper are only in the framework of Hilbert spaces, 
the methodology used here is applicable to the more general reflexive Banach spaces, in which the Pontryagin maximum principle for optimal 
 control problems  
 is  valid. 
Using the solvability of algebraic Riccati equation in Banach spaces (see e.g. \cite{K1}) instead of that in Hilbert spaces,
a similar procedure can be carried out with some slight modifications as what we have done in our analysis for the case of Hilbert spaces.

\item
We have emphasized in the introduction that the final state $y(T)$ in the underlying optimal control problems is assumed to be free.
This raises an obvious question: what happens if the final state is fixed, for instance, $y(T)=y_1$? In finite dimension, the exponential turnpike theorem 
for this question has been established in \cite{TZ1} under suitable Kalman rank condition.  In infinite dimension, however,  from our analysis there is a twofold difficulty for this question.
On one hand, the Pontryagin maximum principle may fail in general. 
On the other hand,
even if the Pontryagin maximum principle is valid under certain finite codimensional  condition on the final state (cf.  \cite[Chapter 4, Theorem 1.6]{LiXunjing}),
however we still do not know how to prove the invertibility of the Lyapunov    operator $E$
constructed in \eqref{731-2}.
Indeed, it 
is generally not invertible. For example, when $$
A=\begin{pmatrix}
-1& 0\\0 &-1
\end{pmatrix},\;\;
B=\begin{pmatrix}
1\\0
\end{pmatrix},\;\;
C=0,\;\;Q=\begin{pmatrix}
1&0\\
0&1
\end{pmatrix},
$$
one can easily check that $E$ is not invertible.

\item We have established our main results for optimal control problems not involving any state or control constraint. Indeed, the presence of constraints complicates the dynamics derived from the Pontryagin maximum principle, and then identifying hyperbolicity features may be very challenging. For instance, control constraints may promote \textit{chattering}, i.e., an infinite number of control switchings over a compact time interval. Note that a turnpike phenomenon is suspected in \cite{Zhu}, while chattering also occurs.

Consider for instance the linear quadratic optimal control problem $(LQ^T)$ considered in Section \ref{sec_periodic}. If we impose the simple control constraint $u\geq 0$, then the application of the Pontryagin maximum principle leads to the same extremal equations \eqref{j152}, and to the extremal control
$$
u(t) = \max\left( 0, u_d(t)+Q^{-1}B^*\lambda^T(t) \right).
$$
Since the extremal control is not smooth anymore, it is not clear how to analyze and use hyperbolicity properties of the extremal system, by linearization around the equilibrium corresponding to the optimal steady-state. Oscillations may indeed occur (with possible chattering as mentioned above) and make the dynamical study complex.

It is therefore challenging to consider the exponential turnpike property for 
optimal control problems with mixed state and/or control constraints, particularly including the time and norm optimal control problems for heat equations, at least, by using the approach developed in this paper, consisting of linearizing the extremal equations coming from the application of the Pontryagin maximum principle.

Instead, an interesting alternative consists of using dissipativity properties of the control system. We refer to \cite{Grune1} for investigating the exponential turnpike property for a class of strictly dissipative discrete-time systems in finite dimension. Even in the presence of constraints, the turnpike property may be analyzed within the viewpoint of strict dissipativity (see \cite{Grune1,Faulwasser1,G3,TZZ1}). Such an analysis may reveal some relationships between strictly dissipativity and hyperbolicity, which are now two methodologies used in the literature to study the  exponential turnpike property.
In \cite{TZZ1}, we provide a comparison between these two approaches which yield results of different natures.

\item Our Theorem \ref{731-10} has been proved under the smallness assumption \eqref{727-88}, because our approach consists of linearizing the optimality system near the optimal steady-state that is also an equilibrium point of the Pontryagin equations. Hence, this linearized system does not reflect what may happen far from this equilibrium point, and this is why the smallness condition \eqref{727-88} is then required in the proof.
Establishing a global result, without such smallness assumptions, may certainly be done, but at the price of having a good knowledge of the \emph{global} dynamics.
Note that, in \cite{Rapaport,Rapaport2}, the authors study in a specific context the optimality status of several turnpikes that are in competition. 
Besides, dissipativity properties of the global dynamics (see \cite{Grune1,Faulwasser1,G3,TZZ1}) may also be a route to deriving global turnpike properties.

\item Turnpike issues can be explored as well for shape optimization problems, by raising the question of whether optimal designs  of a shape optimization problem for evolution systems approximate to those of an optimal steady-state one, as the time horizon is large enough.  We refer to \cite{AMP} for an example of such a large time behavior for the two-phase optimal design for the heat equation, by using relaxation and homogenization. 
We notice that whether the admissible shapes  depend on time or not makes a huge difference in the framework of shape optimization problems for evolution systems. 
\end{enumerate}


\section{Proofs}
For optimal control problems governed by infinite-dimensional evolution systems in which the Pontryagin maximum principle can be applied, we obtain a Hamiltonian system (two-point boundary value problem) coupling the optimal state and the associated adjoint state. In the proofs, we develop a dichotomy transformation acting on the solutions of the operator algebraic Riccati and Lyapunov type equations, in order to ``decouple" the Hamiltonian system to a block-diagonal one, containing a contracting part (which is stable) and an expanding part (which is unstable).  As a byproduct, we obtain a quantitative description of the limiting behavior, interpreted as the exponential turnpike property, of the optimal solutions of the original optimal control problem, as the time horizon is large enough.  

Since the proofs of Theorems~\ref{periodictarget} and \ref{thm1LQ} are much easier than that of Theorem~\ref{731-10}, for the reader's convenience, we first give their proofs in Section~\ref{1252} and then we provide the proof of Theorem~\ref{731-10} in Section~\ref{1251}.

\subsection{Proofs of Theorems~\ref{periodictarget} and \ref{thm1LQ}}\label{1252}
First of all, note that the extremal solution $(y_s,\lambda_s)$ of the problem $(P_s)$, solution of \eqref{6142}, is an equilibrium point for the Hamiltonian system \eqref{j152} with $y_d(\cdot)\equiv y_d$ and $u_d(\cdot)\equiv u_d$.
We are going to prove that this equilibrium is a saddle point for the system \eqref{j152}, thus yielding the turnpike property.
Setting
\begin{equation}\label{add1}
\delta y(t) = y^T(t)-y_s,\quad \delta\lambda(t)=\lambda^T(t)-\lambda_s,\;\;\;t\in[0,T],
\end{equation}
we get from \eqref{j152} and \eqref{6142} that
\begin{equation}\label{6143}
\frac{d}{dt}
\begin{pmatrix}
 \delta  y(t)\\ \delta \lambda(t)
\end{pmatrix}
=M
\begin{pmatrix}
\delta y(t)\\ \delta \lambda(t)
\end{pmatrix},
\;\;t\in[0,T],
\end{equation}
where $M:D(A)\times D(A^*)\rightarrow X\times X$ is the linear unbounded operator block defined by
\begin{equation}\label{eqxlambda}
M=
\begin{pmatrix}
A & BQ^{-1}B^* \\ C^*C & -A^*
\end{pmatrix}.
\end{equation}

\begin{lemma}\label{dichotomy}
Assume that $(A,B)$ is exponentially stabilizable and that $(A,C)$ is exponentially detectable. Then $M$ is block-diagonalizable and boundedly invertible.
Moreover, the system \eqref{6143} can be decoupled by a bounded linear transformation.
\end{lemma}

\begin{proof}
 It is well known from \cite[Theorem 4.4]{zabczyk} that, under these assumptions, 
the operator algebraic Riccati equation 
\begin{equation}\label{riccati}
A^*P+PA-PBQ^{-1}B^*P+C^*C=0,
\end{equation}
has a unique nonnegative definite self-adjoint operator solution $P\in L(X,X)$. Moreover, the 
operator $A-BQ^{-1}B^*P$ generates an exponentially stable $C_0$ semigroup $(S(t))_{t\geq 0}$, satisfying
\begin{equation}\label{est}
\|S(t)\|_{L(X,X)}\leq ce^{-\nu t}, \;\;t\geq0,
\end{equation}
for some constants $c>0$ and $\nu>0$. As a consequence (see \cite[Theorem 3.1]{zabczyk} for instance), the spectral abscissa of $A-BQ^{-1}B^*P$ satisfies
\begin{equation*}\label{spectral}
\sup\big\{\text{Re}\,\lambda\ \mid\ \lambda\in \sigma (A-BQ^{-1}B^*P)\big\}\leq -\nu<0,
\end{equation*}
and thus the operator $A-BQ^{-1}B^*P$  is boundedly invertible.
Since the semigroup $(S(t))_{t\geq0}$ is  exponentially stable, the Lyapunov integral operator $ E\in L(X,X)$ given by
\begin{equation}\label{731-2}
 E= -\int_0^{+\infty} S(t)BQ^{-1}B^*S(t)^*\,dt
\end{equation}
is well defined, and is the solution of the Lyapunov operator equation (see also \cite{zabczyk})
\begin{equation}\label{727-3}
(A-BQ^{-1}B^*P) E+ E(A-BQ^{-1}B^*P)^*-BQ^{-1}B^*=0,
\end{equation}
or equivalently,
$$
2\langle  (A-BQ^{-1}B^*P)Ex,x\rangle_X-\langle BQ^{-1}B^*x,x\rangle_X =0,\;\;\forall x\in D(A).
$$
We now construct a dichotomy transformation in order to decouple the system \eqref{6143}, based on the linear and bounded  operators $P$ and $E$. 
We first define two linear transformations on $X\times X$ by
\begin{equation*}
\mathcal T_1= 
\begin{pmatrix}
I & 0\\ P & I
\end{pmatrix}
\;\;\;\;\text{and}\;\;\;\;
\mathcal T_2= 
\begin{pmatrix}
I & 0\\ -P & I
\end{pmatrix},
\end{equation*}
where $I$ is the identity operator on $X$. Note that
\begin{equation*}\label{6144}
\mathcal T_1\circ\mathcal T_2=
\begin{pmatrix}
I&0\\0& I
\end{pmatrix}
=\mathcal T_2 \circ \mathcal T_1.
\end{equation*}
Since  $P$ solves the Riccati equation \eqref{riccati}, a straightforward computation shows that
\begin{equation}\label{6145}
\mathcal T_1\circ M \circ \mathcal T_2=
\begin{pmatrix}
A-BQ^{-1}B^*P & BQ^{-1}B^*\\
0& -(A-BQ^{-1}B^*P)^*
\end{pmatrix}.
\end{equation}
Setting
\begin{equation*}\label{6178}
\begin{pmatrix}
v(t)\\w(t)
\end{pmatrix}
=\mathcal T_1
\begin{pmatrix}
\delta y(t)\\ \delta \lambda(t)
\end{pmatrix},\;\;t\in[0,T],
\end{equation*}
we infer from  \eqref{6143} and \eqref{6145} that
\begin{equation}\label{727-2}
\frac{d}{dt}
\begin{pmatrix}
v(t)\\ w(t)
\end{pmatrix}
=\begin{pmatrix}
A-BQ^{-1}B^*P & BQ^{-1}B^*\\
0& -(A-BQ^{-1}B^*P)^*
\end{pmatrix}
\begin{pmatrix}
v(t)\\  w(t)
\end{pmatrix},\;\;t\in[0,T].
\end{equation}
Now, we set
\begin{equation*}
\mathcal T_3=
\begin{pmatrix}
I & E\\ 0 & I
\end{pmatrix}
\;\;\;\;\text{and}\;\;\;\;
\mathcal T_4=
\begin{pmatrix}
I & -E\\
0 & I
\end{pmatrix}.
\end{equation*}
Note that
$$\mathcal T_3\circ\mathcal T_4=\begin{pmatrix}
I&0\\0& I
\end{pmatrix}=\mathcal T_4\circ\mathcal T_3.$$
By performing the transformation 
\begin{equation*}\label{727-4}
\begin{pmatrix}
z(t) \\ q(t)
\end{pmatrix}
=\mathcal T_3
\begin{pmatrix}
v(t) \\ w(t)
\end{pmatrix},\;\;t\in[0,T],
\end{equation*}
we infer from \eqref{727-2} and \eqref{727-3} that
\begin{equation}\label{dec}
\frac{d}{dt}
\begin{pmatrix}
z(t) \\ q(t)
\end{pmatrix}
=
\begin{pmatrix}
A-BQ^{-1}B^*P &0 \\  
0 & -(A-BQ^{-1}B^*P)^* 
\end{pmatrix}
\begin{pmatrix}
z(t) \\ q(t)
\end{pmatrix},\;\; t\in[0,T].
\end{equation}
 Finally, we see that $M$ can be block-diagonalized by using the composition transformation $\mathcal T=\mathcal T_3 \circ\mathcal T_1$ given by
\begin{equation*}
\mathcal T=
\begin{pmatrix}
I+ E P & E\\
P      &  I
\end{pmatrix}.
\end{equation*}
The lemma is proved.
\end{proof}

\begin{remark}
Lemma~\ref{dichotomy} implies that the optimality system \eqref{6142} has a unique solution.
\end{remark}


\begin{remark}
In the finite-dimensional case, a dichotomy transformation similar to the one designed above has been used in \cite[Lemma 2.5]{Lukes} and \cite{Sak}. Here, we adapt this dichotomy transformation in the setting of infinite-dimensional Hilbert spaces. We also refer the reader to \cite{WK} for a different dichotomy transformation,  which is, however, based on positive and negative definite solutions of the matrix algebraic Riccati equation. 
\end{remark}

\begin{remark}
The applications of dichotomy transformations for the Hamiltonian system, resulting from the Pontryagin maximum principle, are also well known in the numerical analysis of optimal control problems. For this issue we refer the reader to the brief paper \cite{R}. Note that this uncoupling dichotomy transformation technique could be used as well for the numerical analysis of optimal control problems for partial differential equations.
\end{remark}

The following stability estimate is inspired from \cite[Lemma 3.5]{PZ1}.

\begin{lemma}\label{upperbound}
Assume that $(A,B)$ is exponentially stabilizable and that $(A,C)$ is exponentially detectable. Then, there exists a constant $c>0$ independent of $T$
such that
\begin{equation}\label{7271}
\|y(T)\|_X+\|\lambda(0)\|_X\leq c\big (\|y(0)\|_X+\|\lambda(T)\|_X\big),
\end{equation}
 for any solution $(y(\cdot),\lambda(\cdot))\in C([0,T];X)\times C([0,T];X)$ of the coupled system
\begin{equation}\label{6181}
\left\{
\begin{split}
\dot y(t) &= A y(t)+BQ^{-1}B^*\lambda(t),\\
\dot\lambda(t)&= C^*Cy(t)-A^*\lambda(t),\;\;\;t\in[0,T].
\end{split}
\right.
\end{equation}
\end{lemma}

\begin{proof}
Since the pair $(A,C)$ is exponentially detectable, the pair $(A^*,C^*)$ is exponentially stabilizable, and thus there exists a bounded linear operator $K\in L(X,V)$ such that the $C_0$ semigroup  generated by $A^*+C^*K$ is exponentially stable.  Let $\varphi(\cdot)\in C([0,T];X)$ be the unique solution of
\begin{equation*}\left\{
\begin{split}
&\dot{\varphi}(t)=-(A^*+C^*K)\varphi(t),\;\;t\in[0,T],\\
&\varphi(T)=y(T).
\end{split}\right.
\end{equation*}
It follows from the exponential decay of the $C_0$ semigroup  that there is a constant $c_1>0$ (independent of $T$) such that 
\begin{equation}\label{decay}
\|\varphi(0)\|_X\leq c_1\|y(T)\|_X
\end{equation}
and
\begin{equation}\label{de1}
\int_0^T\|\varphi(t)\|_X^2\,dt\leq c_1\|y(T)\|_X^2.
\end{equation}
Multiplying by $\varphi(t)$ the first equation in \eqref{6181} and integrating the result over $t\in[0,T]$, we get
\begin{equation}\label{id}
\|y(T)\|_X^2=\langle y(0),\varphi(0)\rangle_X +\int_0^T \left( \big\langle BQ^{-1}B^*\lambda(t),\varphi(t)\big\rangle_X -\big\langle K^*Cy(t),\varphi(t)\big\rangle_X \right) dt.
\end{equation}
By the Cauchy-Schwarz inequality and \eqref{de1},  we see that
\begin{equation*}
\begin{split}
\int_0^T\big| \langle BQ^{-1}B^*\lambda(t),\varphi(t)\rangle_X\big|\,dt&\leq  \Big(\int_0^T\|BQ^{-1}B^*\lambda(t)\|_X^2\,dt\Big)^{1/2}\Big(\int_0^T\|\varphi(t)\|_X^2\,dt\Big)^{1/2}\\
&\leq c_2\|y(T)\|_X \|BQ^{-1/2}\|_{L(U,X)}
 \Big(\int_0^T\|Q^{-1/2}B^*\lambda(t)\|_U^2\,dt\Big)^{1/2}
\end{split}
\end{equation*}
and
\begin{equation*}
\begin{split}
\int_0^T\big|\langle K^*Cy(t),\varphi(t)\rangle_X \big|\,dt
&\leq \|K^*\|_{L(V,X)}\Big(\int_0^T\|Cy(t)\|_V^2\,dt\Big)^{1/2}\Big(\int_0^T\|\varphi(t)\|_X^2\,dt\Big)^{1/2}\\
&\leq c_3\|K^*\|_{L(V,X)}\|y(T)\|_X\Big(\int_0^T\|Cy(t)\|_V^2\,dt\Big)^{1/2}.
\end{split}\end{equation*}
These two inequalities, together with \eqref{id} and \eqref{decay}, imply that 
\begin{equation}\label{6182}
\begin{split}
\| y(T)\|_X^2&\leq c_4\Big(\int_0^T\|Cy(t)\|_V^2+\|Q^{-1/2}B^*\lambda(t)\|_U^2\,dt+\| y(0)\|_X^2\Big) ,
\end{split}
\end{equation}
for some positive constant $c_4$ (independent of $T$).

Similarly, since the pair $(A,B)$ is exponentially stabilizable, we obtain from the second equation in \eqref{6181} that
\begin{equation}\label{6171}
\|\lambda(0)\|_X^2 \leq c_5\Big(\int_0^T\|Q^{-1/2}B^*\lambda(t)\|_U^2+\|C y(t)\|_V^2\,dt+\|\lambda(T)\|_X^2\Big) ,
\end{equation}
for some positive  constant $c_5$ (independent of $T$). 

Let $c_6=\max(c_4,c_5)$.
Next, multiplying by $\lambda(t)$ the first equation in \eqref{6181} and
by $y(t)$ the second equation in \eqref{6181}, 
and then integrating 
over $t\in[0,T]$, we get from the Cauchy-Schwarz inequality that
\begin{equation*} \label{6172}
\begin{split}
\int_0^T\left( \|C y(t)\|_V^2+\|Q^{-1/2}B^*\lambda(t))\|_U^2\right)dt&=\big\langle y(T),\lambda(T)\big\rangle_X-\big\langle  y(0),\lambda(0)\big\rangle_X\\
&\leq \|\lambda(T)\|_X\| y(T)\|_X+\|y(0)\|_X\|\lambda(0)\|_X\\
&\leq c_6\|\lambda(T)\|_X^2+\frac{1}{4c_6}\| y(T)\|_X^2+ c_6\| y(0)\|_X^2+
\frac{1}{4c_6}\|\lambda(0)\|_X^2.
\end{split}
\end{equation*}
This, along with \eqref{6171} and \eqref{6182}, implies that
$$\int_0^T\left( \|C y(t)\|_V^2+\|Q^{-1/2}B^*\lambda(t)\|_U^2\right)dt\leq c_7\big(\|y(0)\|^2_X+\|\lambda(T)\|^2_X\big) ,$$
for some positive constant $c_7$ independent of $T$.
Using \eqref{6171} and \eqref{6182} again, this leads to \eqref{7271} and completes the proof.
\end{proof}
\bigskip

Using Lemmas~\ref{dichotomy} and \ref{upperbound}, we are now in position to prove Theorem~\ref{thm1LQ}.

\begin{proof}[\textbf{Proof of Theorem~\ref{thm1LQ}}]
Let $\delta y(\cdot)$ and $\delta \lambda(\cdot)$ be defined by \eqref{add1}. 
By using the same dichotomy transformation 
\begin{equation}\label{add2}
\begin{pmatrix}
z(t) \\ q(t)
\end{pmatrix}
=\begin{pmatrix}
I+E P & E\\
P      &  I
\end{pmatrix}
\begin{pmatrix}
\delta y(t) \\ \delta\lambda (t)
\end{pmatrix},\;\;t\in[0,T],
\end{equation}
as in the proof of Lemma~\ref{dichotomy}, we obtain a decoupled evolution system \eqref{dec}.
Consequently, we have
\begin{equation}\label{add3}
z(t)=S(t)z(0) \;\;\text{and}\;\; q(t)=S(T-t)^*q(T),\;t\in[0,T],
\end{equation}
where the $C_0$ semigroup $(S(t))_{t\geq 0}$ is generated by the operator $A-BQ^{-1}B^*P$ satisfying the exponential decay estimate \eqref{est},
 and $(S(t)^*)_{t\geq 0}$ is its adjoint $C_0$ semigroup.

Note that $u^T(t)-u_s=Q^{-1}B^*\delta\lambda(t)$ for a.e. $t\in[0,T]$. To derive the estimate \eqref{6175},  according to \eqref{add2} and \eqref{add3}, it suffices to 
show that the norms of $z(0)$ and $q(T)$ have an upper bound which is independent of $T$. Note from \eqref{add2}
 that
\begin{equation*} 
z(0)=(I+ E P)\delta y(0)+E\delta\lambda(0),\;\;\;\;\;\;q(T)=P\delta y(T)+\delta\lambda(T).
\end{equation*}
Therefore, in order to complete the proof, it suffices to give an upper bound for $\|\delta y(T)\|_X$, as well as for $\|\delta\lambda(0)\|_X$, which follows from Lemma~\ref{upperbound}.  
The theorem is proved.
\end{proof}

Before turning to the proof of Theorem~\ref{periodictarget}, by using Lemma~\ref{dichotomy} we are now in a position to formulate explicitly the optimal extremal  triple $(y^\Pi(\cdot), u^\Pi(\cdot),\lambda^\Pi(\cdot))$ for  the periodic optimal control problem $(LQ^\Pi)$, thus proving the contents of Remark \ref{explicit}.
 
\begin{lemma}\label{lemma_explicit}
Under the assumptions of Theorem~\ref{periodictarget}, the (unique) optimal\; 
$\Pi$-periodic extremal triple \\ $(y^\Pi(\cdot),\lambda^\Pi(\cdot),u^\Pi(\cdot))$  of the problem $(LQ^\Pi)$ is given by 
$$
y^\Pi(t)=z(t)- E q(t),\quad \lambda^\Pi(t)=-Pz(t)+(I+PE)q(t),\quad
u^\Pi(t)=u_d(t)+Q^{-1}B^*\lambda^\Pi(t),$$
for almost every $t\in[0,\Pi]$.
Here, $P$ and $ E$ are accordingly linear bounded operators defined by \eqref{riccati} and \eqref{731-2},  $(z(t),q(t))$, $t \in[0,\Pi]$, are the periodic trajectories given by \eqref{tian6} and \eqref{tian7} below, respectively.
\end{lemma}

\begin{proof}
Using the dichotomy transformation already used in the proof of Lemma~\ref{dichotomy},
\begin{equation}\label{jin4}
\begin{pmatrix}
z(t) \\ q(t)
\end{pmatrix}
=\begin{pmatrix}
I+E P & E\\
P      &  I
\end{pmatrix}
\begin{pmatrix}
y^\Pi(t) \\ \lambda^\Pi (t)
\end{pmatrix},\;\;t\in[0,\Pi],
\end{equation}
we can uncouple the system \eqref{xiaoextremalsystLQ} to 
\begin{equation}\label{tian1}
\dot z(t)=(A-BQ^{-1}B^*P)z(t)+\big((I+ E P)Bu_d(t)- EC^*Cy_d(t)\big),\;\;t\in[0,\Pi],
\end{equation}
and
\begin{equation}\label{tian2}
\dot q(t)=-(A-BQ^{-1}B^*P)^*q(t)+\big(PBu_d(t)-C^*Cy_d(t)\big),\;\;t\in[0,\Pi].
\end{equation}
Now, using the periodic boundary conditions, we are going to determine the initial data $z(0)$ and $q(T)$ for the equations \eqref{tian1} and \eqref{tian2}, respectively.  
It follows from \eqref{jin4} and from the periodic condition in \eqref{xiaoextremalsystLQ} that 
\begin{equation}\label{jin6}
z(0)=z(\Pi).
\end{equation}
By the Duhamel formula, we get from \eqref{tian1} that
\begin{equation}\label{jin7}
z(\Pi)=S(\Pi)z(0)+\int_0^\Pi S(\Pi-\tau)\big((I+ E P)Bu_d(\tau)- EC^*Cy_d(\tau)\big)\,d\tau,
\end{equation}
where $(S(t))_{t\geq 0}$ is the $C_0$ semigroup generated by the operator $A-BQ^{-1}B^*P$ satisfying the exponential decay estimate \eqref{est}.
It follows from \cite[Corollary 2.1]{Barbu} that the operator $I-S(\Pi)$ is boundedly invertible.  Hence,  we obtain from \eqref{jin6} and \eqref{jin7} that
\begin{equation*}
z(0)=(I-S(\Pi)\big)^{-1}\int_0^\Pi S(\Pi-\tau)\big((I+ E P)Bu_d(\tau)- EC^*Cy_d(\tau)\big)\,d\tau.
\end{equation*}
Therefore, we get
\begin{multline}\label{tian6}
z(t)=S(t)(I-S(\Pi)\big)^{-1}\int_0^\Pi S(\Pi-\tau)\big((I+ E P)Bu_d(\tau)- EC^*Cy_d(\tau)\big)\,d\tau
\\+\int_0^tS(t-\tau)\big((I+E P)Bu_d(\tau)-EC^*Cy_d(\tau)\big)\,d\tau,\;\;t \in[0,\Pi].
\end{multline}
Applying similar arguments to \eqref{tian2}, we also obtain that  
\begin{multline}\label{tian7}
q(t)=S(\Pi-t)^*(I-S(\Pi)^*\big)^{-1}\int_0^\Pi S(\Pi-\tau)^*\big(-PBu_d(\Pi-\tau)+C^*Cy_d(\Pi-\tau))\big)\,d\tau
\\+\int_0^{\Pi-t}S(\Pi-t-\tau)^*\big(-PBu_d(\Pi-\tau)+C^*Cy_d(\Pi-\tau)\big)\,d\tau,\;\;t \in[0,\Pi].
\end{multline}
Noting from the transformation \eqref{jin4} that  
\begin{equation*}
\begin{pmatrix}
y^\Pi(t) \\ \lambda^\Pi(t)
\end{pmatrix}
=
\begin{pmatrix}
I & - E \\ -P & I+P E
\end{pmatrix}
\begin{pmatrix}
z(t) \\ q(t)
\end{pmatrix},\;\; t\in[0,\Pi],
\end{equation*}
the lemma follows.
\end{proof}

We can now complete the proof of Theorem~\ref{periodictarget}.

\begin{proof}[\textbf{Proof of Theorem~\ref{periodictarget}}]
Setting
\begin{equation*}
\delta y(t)=y^T(t)-y^\Pi(t),\;\;\delta\lambda(t)=\lambda^T(t)-\lambda^\Pi(t),\;\;\delta u(t)=u^T(t)-u^\Pi(t),\;\;
t\in[0,T],
\end{equation*}
we get from \eqref{j152} and \eqref{xiaoextremalsystLQ}  that
\begin{equation}\label{xiao1012}
\left\{
\begin{split}
\frac{d\delta y(t) }{dt}&= A\delta y(t)+BQ^{-1}B^*\delta\lambda(t), \;\; \;t\in[0,T],\\
\frac{d\delta\lambda(t)}{dt} &= C^*C\delta y(t)-A^*\delta\lambda(t),\;\;\;t\in[0,T],\\
\end{split}
\right.
\end{equation}
with the terminal conditions
\begin{equation*}\label{731-3}
\delta y(0)=y_0-y^\Pi(0),\;\;\;\delta\lambda(T)=-\lambda^\Pi\Big(T-\left[\frac T\Pi\right]\Pi\Big).
\end{equation*}
Here, $[x]$ denotes the largest integer less than or equal to $x$.
Using the dichotomy transformation
\begin{equation}\label{731-4}
\begin{pmatrix}
v(t) \\ w(t)
\end{pmatrix}
=\begin{pmatrix}
I+E P &  E\\
P      &  I
\end{pmatrix}
\begin{pmatrix}
\delta y(t) \\ \delta\lambda (t)
\end{pmatrix},\;\;t\in[0,T],
\end{equation}
where $P$ and $ E$ are accordingly given by \eqref{riccati} and \eqref{731-2},
we transform the system \eqref{xiao1012} to  
\begin{equation*}
\frac{d}{dt}
\begin{pmatrix}
v(t) \\ w(t)
\end{pmatrix}
=
\begin{pmatrix}
A-BQ^{-1}B^*P &0 \\  
0 & -(A-BQ^{-1}B^*P)^* 
\end{pmatrix}
\begin{pmatrix}
v(t) \\ w(t)
\end{pmatrix},\;\; t\in[0,T].
\end{equation*}
Therefore,
\begin{equation}\label{731-5}
v(t)=S(t)v(0) \;\;\text{and}\;\; w(t)=S(T-t)^*w(T),\;t\in[0,T],
\end{equation}
where $(S(t))_{t\geq 0}$ is the exponentially stable $C_0$ semigroup generated by $A-BQ^{-1}B^*P$.
In particular, by \eqref{731-4}, we have
\begin{equation}\label{9181}
v(0)=(I+E P)\delta y(0)+ E\delta \lambda(0)\;\;\;\;\text{and}\;\;\;
w(T)=\delta \lambda(T)+P\delta y(T).
\end{equation}
By Lemma~\ref{upperbound}, we infer that the stability estimate 
$$
\|\delta y(T)\|_X+\|\delta\lambda(0)\|_X\leq c\big(\|\delta y(0)\|_X+\|\delta\lambda(T)\|_X\big),
$$
holds true for some positive constant $c$ independent of $T$.
This estimate, together with \eqref{9181}, \eqref{731-5}, as well as the bounded invertibility of the dichotomy transformation \eqref{731-4}, lead to the estimate
$$
\|\delta y(t)\|_X+\|\delta \lambda(t)\|_X\leq c(\|\delta y(0)\|_X+\|\delta\lambda(T)\|_X)(e^{-\nu t}+e^{-\nu (T-t)}),\;\;\forall t\in[0,T],
$$
for some positive constants $c$ and $\nu$ independent of $T$. The theorem is proved.
\end{proof}

\subsection{Proof of Theorem~\ref{731-10}}\label{1251}
We follow the arguments of the proofs of Lemmas~\ref{dichotomy}, \ref{upperbound} and of Theorem~\ref{thm1LQ}, as well as those of the proof of \cite[Theorem 1]{TZ1}.

According to the assumptions, the Hamiltonian $H$ is twice continuously Fr\'echet differentiable in $X\times X\times U$, and thus for any $(y,\lambda, u)\in X\times X\times U$, there exists a constant $R_0>0$ such that,
for any $0<R< R_0$,
  the asymptotic expansion
formula 
\begin{equation*}
\begin{split}
H_\bigstar(y+\delta y,\lambda+&\delta\lambda,u+\delta u)-H_\bigstar(y,\lambda,u)\\
&=H_{\bigstar y}(y,\lambda,u)\delta y
+H_{\bigstar \lambda}(y,\lambda,u)\delta \lambda+H_{\bigstar u}(y,\lambda,u)\delta u+  o(\delta y,\delta \lambda,\delta u),
\end{split}
\end{equation*}
holds for any small perturbation $(\delta y,\delta \lambda,\delta u)$ verifying 
\begin{equation}\label{727-6}
\|\delta y\|_X+\|\delta\lambda\|_X+\|\delta u\|_U\leq R.
\end{equation}
Here the symbol $\bigstar$ stands for the Fr\'echet derivative of $H$ with respect to
the variable $y$ or $\lambda$ or $u$, and $o(\delta h)$ is the remaining higher-order terms with respect to $\delta h$.

We start by defining perturbations of $(y^T(\cdot),\lambda^T(\cdot),u^T(\cdot))$ with respect to $(y_s,\lambda_s,u_s)$, by
\begin{equation*}
\delta y(t)=y^T(t)-y_s,\;\;\delta \lambda(t)=\lambda^T(t)-\lambda_s,\;\;\delta u(t)=u^T(t)-u_s,\;\;t\in[0,T].
\end{equation*}
Under the assumptions of the theorem, we  make the following \textit{a priori}  hypotheses: 
\begin{equation}\label{722-6}
\begin{split}
&(i).\;\;\;\;\;\;\|(\delta y(t),\delta \lambda(t),\delta u(t))\|_{X\times X\times U}\leq R\;\;\text{for a.e.}\;\;t\in[0,T];\\
&(ii).\;\;\;\;\;\;\int_0^T \left( \|\delta y(t)\|_X^2+\|\delta \lambda(t)\|_X^2\right) dt\leq R,
\end{split}
\end{equation}
where the positive constant $R$ is sufficiently small.
These two hypotheses will be verified a posteriori  at the end by an appropriate choice of smallness constraint \eqref{727-88}.
First of all, it follows from \eqref{721-2} and \eqref{721-3} that, at the point $(y_s,\lambda_s,u_s)$, 
\begin{equation*}\label{721-4}
H_{uy}\delta y(t)+H_{u\lambda}\delta\lambda(t)+H_{uu}\delta u(t)+o(\delta y(t),\delta \lambda(t),\delta u(t))=0
\;\;\text{for a.e.}\;\;t\in[0,T].
\end{equation*}
Since $H_{uu}^{-1}$ is bounded, we get
\begin{equation}\label{721-61c}
\delta u(t)=-H^{-1}_{uu}\big(H_{uy}\delta y(t)+H_{u\lambda}\delta \lambda(t)\big)+o(\delta y(t),\delta \lambda(t)).
\end{equation}
Therefore, using \eqref{727-10} and \eqref{727-11}, we infer from \eqref{721-1}, \eqref{721-5} and \eqref{721-61c} that
\begin{equation}\label{721-8}
\frac{d}{dt}
\begin{pmatrix}
\delta y(t)\\ \delta\lambda(t)
\end{pmatrix}
=
\begin{pmatrix}
\mathcal A & -H_{\lambda u}H_{uu}^{-1}H_{u\lambda}\\
\mathcal{C}^*\mathcal C& -\mathcal{A}^*
\end{pmatrix}
\begin{pmatrix}
\delta y(t)\\ \delta\lambda(t)
\end{pmatrix}
+o(\delta y(t),\delta\lambda (t)),
\end{equation}
with the two-point boundary conditions
\begin{equation}\label{10221}
\delta y(0)=y_0-y_s,\;\;\;\;\delta\lambda(T)=-\lambda_s.
\end{equation}
Note that the principal part of the equation \eqref{721-8} has the same structure as the operator block $M$ given by
\eqref{eqxlambda}. As already mentioned, in comparison with the proof of Theorem~\ref{thm1LQ}, the difficulty   is here to  deal carefully with the  higher-order remaining terms in \eqref{721-8}.

Next, by using arguments similar to those in the proof of Lemma~\ref{dichotomy}, we are going to uncouple the principal part of the linearized system
\eqref{721-8}.
Let $ \mathcal P\in L(X,X)$ be the unique nonnegative definite, self-adjoint operator solution of the operator algebraic Riccati equation (see the same reasonings for \eqref{riccati})
\begin{equation*}\label{721-6}
\mathcal A^* \mathcal P+ \mathcal P\mathcal A+\mathcal C^*\mathcal C+ \mathcal PH_{\lambda u}H_{uu}^{-1}H_{u\lambda} \mathcal P=0.
\end{equation*}
Moreover,  the operator $\mathcal A+H_{\lambda u}H_{uu}^{-1}H_{u\lambda} \mathcal P$ generates an exponentially stable $C_0$ semigroup 
$(\mathcal S(t))_{t\geq 0}$ in $X$, i.e.,
\begin{equation}\label{721-10}
\|\mathcal S(t)\|_{L(X,X)}\leq c_1e^{-\nu t}\;\;\;\text{for}\;\;t\geq0,
\end{equation}
for some positive constants $c_1$ and $\nu$. We then define the linear bounded selfadjoint operator on $X$
\begin{equation*}
\mathcal E= \int_0^{+\infty} \mathcal S(t) H_{\lambda u}H_{uu}^{-1}H_{u\lambda}
\mathcal S(t)^*\,dt.
\end{equation*}
Using the dichotomy transformation
\begin{equation}\label{721-11}
\begin{pmatrix}
v(t) \\ w(t)
\end{pmatrix}
=\begin{pmatrix}
I+\mathcal E  \mathcal P & \mathcal E\\
 \mathcal P      &  I
\end{pmatrix}
\begin{pmatrix}
\delta y(t)\\ \delta\lambda(t)
\end{pmatrix}\;\;\;\;\;t\in[0,T],
\end{equation}
we transform the system \eqref{721-8} to 
\begin{equation*}\label{10301}
\frac{d}{dt}
\begin{pmatrix}
v(t)\\ w(t)
\end{pmatrix}
=\begin{pmatrix}
\mathcal A+H_{\lambda u}H_{uu}^{-1}H_{u\lambda}  \mathcal P&0\\
0 & -\big(\mathcal A+H_{\lambda u}H_{uu}^{-1}H_{u\lambda} \mathcal  P\big)^*
\end{pmatrix}
\begin{pmatrix}
v(t)\\ w(t)
\end{pmatrix}
+o(v(t),w(t)).
\end{equation*}
Solving the first equation in forward time and the second equation in backward time, and using \eqref{721-10}, we get
\begin{equation}\label{722-2}
\|v(t)\|_X+\|w(t)\|_X\leq 4c_1(e^{-\nu t/2}\|v(0)\|_X+e^{-\nu(T-t)/2}\|w(T)\|_X),
\end{equation}
for every $t\in[0,T]$.

The remainder of the proof consists of determining the values $(\|v(0)\|_X,\|w(T)\|_X)$ from the terminal conditions \eqref{10221}.
We first claim that the inequality
\begin{equation}\label{722-4}
\|\delta y(T)\|_X+\|\delta\lambda(0)\|_X\leq c_2\big(\|y_0-y_s\|_X+\|\lambda_s\|_X\big)+o(R),
\end{equation}
holds for  some constant $c_2$ independent of $T$ (the proof of this claim is postponed to the end).  
It follows from the transformation \eqref{721-11}  that
\begin{equation*}
v(0)=(I+\mathcal E  \mathcal P)\delta y(0)+\mathcal E\delta \lambda(0),\;\;\;\; w(T)= \mathcal P\delta y(T)+\delta \lambda(T).
\end{equation*}
Hence, we infer from  \eqref{721-11}, \eqref{722-2} and \eqref{722-4} that
\begin{equation}\label{727-886}
\|\delta y(t)\|_X+\|\delta \lambda(t)\|_X
\leq \Big(c_3(\|y_0-y_s\|_X+\|\lambda_s\|_X)+o(R) \Big)\big(e^{-\nu t/2}+e^{-\nu(T-t)/2}\big)
\end{equation}
for every $t\in[0,T]$, with some constant $c_3$ independent of $T$.
This, together with \eqref{721-61c}, leads to a similar estimate for $\delta u(\cdot)$.
Therefore, there exists $\varepsilon >0$ such that the estimates \eqref{722-6} hold true whenever the inequality \eqref{727-88} holds.
As a consequence of \eqref{727-886}, the exponential turnpike property \eqref{727-18} is proved.

Finally, let us prove the claim \eqref{722-4}, which is analogous to that in Lemma~\ref{upperbound}.  
For the first equation in \eqref{721-8}, since $(\mathcal A^*, \mathcal C^*)$ is exponentially detectable, by the same reasoning as for \eqref{6182}, there exists a $c_0>0$ independent of $T$ such that
\begin{equation}\label{722-10}
\|\delta y(T)\|_X^2 \leq 
c_0\Big(\int_0^T \left( \|\mathcal C\delta y(t)\|_V^2+\|H_{uu}^{-1/2}H_{u\lambda}\delta \lambda(t)\|_U^2\right) dt+\|\delta y(0)\|_X^2\Big)
+o(R).
\end{equation}
Similarly, we obtain
\begin{equation}\label{722-11}
\|\delta\lambda (0)\|_X^2\leq c_0\Big(\int_0^T \left( \|H^{-1/2}_{uu}H_{u\lambda }\delta\lambda(t)\|_U^2+\|\mathcal C\delta y(t)\|_V^2\right) dt
+\|\delta\lambda(T)\|_X^2\Big)+o(R).
\end{equation}
Multiplying by $\delta\lambda(t)$ the first equation in \eqref{721-8} and by $\delta y(t)$ the second equation in \eqref{721-8},
and then integrating over $t\in[0,T]$, we get that
\begin{multline*}
\int_0^T \left( \|H^{-1/2}_{uu}H_{u\lambda }\delta\lambda(t)\|_U^2+\|\mathcal C \delta y(t)\|_V^2\right) dt
\leq \|\delta\lambda(T)\|_X\|\delta y(T)\|_X+\|\delta y(0)\|_X\|\delta\lambda(0)\|_X+o(R)\\
\leq \sigma\big(\|\delta\lambda(0)\|_X^2+\|\delta y(T)\|_X^2\big)+\frac{1}{4\sigma}\big(\|\delta y(0)\|_X^2+\|\delta\lambda (T)\|_X^2\big)
+o(R)
\end{multline*}
for any real number $\sigma>0$.
This, together with \eqref{722-10} and \eqref{722-11}, implies that
\begin{multline*}
\int_0^T \left( \|H^{-1/2}_{uu}H_{u\lambda }\delta\lambda(t)\|_U^2+\|\mathcal C\delta y(t)\|_V^2\right) dt
\leq 2c_0\sigma \int_0^T \left( \|H^{-1/2}_{uu}H_{u\lambda }\delta\lambda(t)\|_U^2+\|\mathcal C\delta y(t)\|_V^2\right) dt \\
+\Big(\frac{1}{4\sigma}+2c_0\sigma\Big)\big( \|\delta y(0)\|_X^2+\|\delta\lambda (T)\|_X^2\big)+o(R).
\end{multline*}
Choosing $\sigma=\frac{1}{4c_0}$, we get that
\begin{equation*}
\int_0^T \left( \|H^{-1/2}_{uu}H_{u\lambda }\delta\lambda(t)\|_U^2+\|\mathcal C\delta y(t)\|_V^2\right) dt
\leq (2c_0+1)\big( \|\delta y(0)\|_X^2+\|\delta\lambda (T)\|_X^2\big)+o(R).
\end{equation*}
This, along with \eqref{10221}, \eqref{722-10} and \eqref{722-11}, implies and completes the proof of the claim \eqref{722-4}.

\bigskip\bigskip

\noindent \textbf{Acknowledgment}.
The authors acknowledge the financial support by the grant FA9550-14-1-0214 of the EOARD-AFOSR. The second author was partially supported by the National Natural Science Foundation of China under grants 11501424 and 11371285.
The third author was partially supported by the Advanced Grant DYCON (Dynamic Control) of the European Research Council Executive Agency, FA9550-15-1-0027 of AFOSR,  the MTM2014-52347 Grant of the MINECO (Spain) and ICON of the French ANR.


\bigskip


\begin{thebibliography}{99}

\bibitem{Kokotovic} B. D. O. Anderson, P. V. Kokotovic, Optimal control problems over large time intervals, Automatica 23 (1987), 355-363.

\bibitem{AMP}  G. Allaire, A. M\"unch,  F. Periago, Long time behavior of a two-phase optimal design for the heat equation, SIAM J. Control and Optim. 48 (2010), 5333-5356.

\bibitem{AL} Z. Artstein, A. Leizarowitz, 
Tracking periodic signals with the overtaking criterion, IEEE Transactions on Automatic Control 30 (1985), 1123-1126.  

\bibitem{Barbu} V. Barbu, N. H. Pavel, Periodic optimal control in Hilbert space, Appl Math Optim 33 (1996), 169-188.

\bibitem{BardosLebeauRauch}
C. Bardos, G. Lebeau, J. Rauch,
Sharp sufficient conditions for the observation, control, and stabilization of waves from the boundary,
SIAM J. Control Optim. {30} (1992), no. 5, 1024-1065.



\bibitem{CLLP} P.  Cardaliaguet,  J.-M. Lasry,  P.-L. Lions,  A. Porretta,
Long time average of mean field games with a nonlocal coupling, SIAM J. Control Optim.,  51 (2013), 3558-3591.


\bibitem{CHJ} D. A. Carlson, A. Haurie, A. Jabrane,
Existence of overtaking solutions to infinite-dimensional control problems on unbounded time intervals, 
SIAM J. Control Optim. 25 (1987), 1517-1541. 

\bibitem{CarlsonBOOK} D. A. Carlson, A. B. Haurie, A. Leizarowitz, 
Infinite Horizon Optimal Control, 2nd ed. New York: Springer Verlag, 1991.


\bibitem{Grune1} T. Damm, L. Gr\"une,  M. Stieler,  K. Worthmann, An exponential turnpike theorem 
for dissipative discrete time optimal control problems, SIAM J. Control Optim. 52  (2014), 1935-1957.



\bibitem{DorfmanSamuelsonSolow}
R. Dorfman, P.A. Samuelson, R. Solow,
Linear Programming and Economic Analysis,
New York, McGraw-Hill, 1958.


\bibitem{Faulwasser1} T. Faulwasser, M. Korda, C. N. Jones, D. Bonvin, On turnpike and dissipativity properties of continuous-time optimal control problems. arXiv: 1509.07315.

\bibitem{AMPL} 
R. Fourer, D.M. Gay, B.W. Kernighan,
AMPL: A Modeling Language for Mathematical Programming,
Duxbury Press, Second Edition, 2002.

\bibitem{EG} 
E. Gilbert,  Optimal periodic control: a general theory of necessary conditions,  SIAM J. Control Optim. 15 (1977),  717-746. 


 
 \bibitem{G3} L. Gr\"une, M. M\"uller, On the relation between strict dissipativity and turnpike properties, Systems and Control Letters 90 (2016),  45-53.
 
\bibitem{GAB}
 F. Grognard, A.R. Akhmetzhanov, O. Bernard, Periodic optimal control for biomass productivity maximization in a photobioreactor using natural light, Inria Research Report, 
 Project-Teams Biocore, 7929 (2012).
  

\bibitem{GTZ} M. Gugat, E. Tr\'elat, E. Zuazua, Optimal Neumann control for the 1D wave equation: Finite horizon, infinite horizon, boundary tracking terms and the turnpike property. Systems and Control Letters 90 (2016), 61-70.



\bibitem{ItoKunisch}
K. Ito, K. Kunisch,
Lagrange Multiplier Approach to Variational Problems and Applications,
Advances in Design and Control, 15, Society for Industrial and Applied Mathematics (SIAM), Philadelphia, PA, 2008.

\bibitem{K1}
S. Koshkin, Positive semigroups and algebraic Riccati equations in Banach spaces, Positivity 20 (2016), 541-563. 

\bibitem{Lions} J. L.  Lions, Optimal Control for Systems Governed by Partial Differential Equations. Springer, Berlin, 1971.

\bibitem{Lukes} D. L. Lukes, Optimal regulation of nonlinear dynamical systems, SIAM J. Control 7 (1969),  75-100.

\bibitem{LW} H. Lou,  W. Wang,  Turnpike properties of optimal relaxed control problems. Preprint.

\bibitem{LiXunjing} X. Li, J. Yong, Optimal Control Theory for Infinite-dimensional Systems. Systems and Control: Foundations and Applications, Birkh\"auser Boston, Inc., Boston, MA, 1995.

\bibitem{Mc} L. W. McKenzie, Turnpike theorems for a generalized Leontief model, Econometrica 31 (1963), 165-180.

\bibitem{Pa} A. Pazy, Semigroups of Linear Operators and Applications
to Partial Differential Equations, Springer-Verlag, 1983.




 \bibitem{PZ1} A. Porretta, E. Zuazua, Long time versus steady state optimal control,
SIAM J. Control and Optim. 51 (2013), 4242-4273.

\bibitem{PZ2} A. Porretta, E. Zuazua, Remarks on long time versus steady state optimal control. Springer-INdAM, Mathematical Paradigms of Climate Science, P. M. Cannarsa et al. eds, to appear.



\bibitem{R}  A. V. Rao, K. D. Mease, Dichotomic basis approach to solving hypersensitive optimal control problems,  Automatica 35 (1999), 633-642.

\bibitem{Rapaport}
A. Rapaport, P. Cartigny,
Turnpike theorems by a value function approach,
ESAIM: Control Optim. Calc. Var. {10} (2004), 123-141.

\bibitem{Rapaport2}
A. Rapaport, P. Cartigny,
Competition between most rapid approach paths: necessary and sufficient conditions,
J. Optim. Theory Appl. {124} (2005), no. 1, 1-27.

\bibitem{Rockafellar1973}
R. T. Rockafellar,
Saddle points of Hamiltonian systems in convex problems of Lagrange,
J. Optimization Theory Appl. {12} (1973), 367-390.

\bibitem{Sak} N. Sakamoto, Analysis of the Hamilton-Jacobi equation in nonlinear control theory by symplectic geometry, SIAM J. Control Optim., 40 (2002), no. 6,  1924-1937.

\bibitem{Samuelson1976}
P. A. Samuelson, 
The periodic turnpike theorem,
Nonlinear Anal. 1 (1976), no. 1, 3--13. 



\bibitem{T} F. Tr\"oltzsch, Optimal Control of Partial Differential Equations. Theory, Methods and Applications.
Graduate Studies in Mathematics, 112,  American Mathematical Society,  2010. 

\bibitem{Trelatbook}
E. Tr\'elat,
Contr\^ole Optimal : Th\'eorie {\rm\&} Applications,
Vuibert, Collection "Math\'ematiques Concr\`etes", 2005.

\bibitem{TZ1} E. Tr\'elat, E. Zuazua, The turnpike property in finite-dimensional nonlinear optimal control,
 J. Differential Equations 258 (2015), 81-114.
 
 \bibitem{TWZ} E. Tr\'elat, L. Wang, Y. Zhang, Impulse and sampled-data optimal control of heat equations, and error estimates. To appear in SIAM J. Control and Optim..
 
 \bibitem{TZZ1} E. Tr\'elat, C. Zhang, E. Zuazua, Measure-turnpike property for dissipative infinite-dimensional optimal control problems. Preprint.
 
\bibitem{IPOPT} 
A. W\"achter, L.T. Biegler,
On the implementation of an interior-point filter line-search algorithm for large-scale 
nonlinear programming,
Mathematical Programming 106 (2006), 25-57.

 
\bibitem{WK} R. R. Wilde,  P. V.  Kokotovic,  A dichotomy in linear control theory, IEEE Trans. Automat, Control, 17 (1972), 382-383.



\bibitem{zabczyk} J. Zabczyk, Mathematical Control Theory: An Introduction.
Systems and Control: Foundations and Applications.
Birkh\"auser Boston, Inc., Boston, MA, 1992. 

\bibitem{Zanon}
M. Zanon, L. Gr\"une, M. Diehl,
Periodic optimal control, dissipativity and MPC.
Preprint Hal (2015).


\bibitem{Z2} A. J. Zaslavski, Turnpike Properties in the Calculus of Variations and Optimal Control. Vol. 80. Springer, 2006.

\bibitem{Z3}
A. J.  Zaslavski, Turnpike Phenomenon and Infinite Horizon Optimal Control, in: Springer Optimization and Its Applications, vol. 99, Springer, Cham, 2014.

\bibitem{Z4}
A. J. Zaslavski, Turnpike Theory of Continuous-Time Linear Optimal Control Problems, in: Springer Optimization and Its Applications, vol. 104, Springer,
Cham, 2015.

\bibitem{Zhu}
J. Zhu, E. Tr\'elat, M. Cerf,
Minimum time control of the rocket attitude reorientation associated with orbit dynamics,
SIAM J. Control  Optim. {54} (2016), no. 1, 391-422.


\bibitem{Zcon} E. Zuazua, Controllability and observability of partial differential equations: some results and open problems,  in:  Handbook of Differential Equations:
Evolutionary Differential Equations, vol. 3, Elsevier Science, 2006, pp. 527-621.
 
\end{thebibliography}
\end{document}